\documentclass[12pt]{article}
\usepackage{amsmath,amsfonts,amsthm,amscd,upref,amstext}
\usepackage[dvips]{graphicx}
\usepackage{subfigure}
\usepackage{pb-diagram,pb-xy}
\usepackage[all]{xy}
\usepackage[dvips]{color}
\usepackage{xcolor}

\newtheorem{prop}{Proposition}[section]
\newtheorem{thm}[prop]{Theorem}
\newtheorem{cor}[prop]{Corollary}

\newtheorem{defn}[prop]{Definition}
\newtheorem{rem}[prop]{Remark}
\newtheorem{lem}[prop]{Lemma}

\newcommand{\N}{\mathbb{N}}

\numberwithin{equation}{section}

\begin{document}

\title{Artinianness and Finiteness  of Formal Local Cohomology Modules with Respect to a Pair of Ideals}
\author{T.H. Freitas $^{1,}\,$\thanks{Work partially supported by FAPESP-Brazil - Grant
2012/01084-0 and 2013/20723-7. }\,\,\,\,and\,\,\, V. H. Jorge P\'erez $^{2}$
\thanks{Work
partially supported by CNPq-Brazil - Grant
245872/2012-4 and FAPESP-Brazil - Grant
2012/20304-1.  2000 Mathematics Subject
Classification: 13D45. {\it Key words}: Local Cohomology, Formal Local Cohomology, Artinian. }}

\date{}
\maketitle

\noindent $^1$ Universidade de S{\~a}o Paulo -
ICMC, Caixa Postal 668, 13560-970, S{\~a}o Carlos-SP, Brazil ({\it
e-mail: tfreitas@icmc.usp.br }).

\vspace{0.3cm}
\noindent $^2$ Universidade de S{\~a}o Paulo -
ICMC, Caixa Postal 668, 13560-970, S{\~a}o Carlos-SP, Brazil ({\it
e-mail: vhjperez@icmc.usp.br}).

\vspace{0.3cm}
\begin{abstract}
Let $(R,\mathfrak{m})$ be a commutative Noetherian local ring, $M$ be a finitely generated $R$-module and  $\mathfrak{a}$, $I$ and $J$  be ideals of $R$. We investigate the structure of formal local cohomology modules of  $\mathfrak{F}^i_{\mathfrak{a},I,J}(M)$ and $\check{\mathfrak{F}}^i_{\mathfrak{a},I,J}(M)$  with respect to a pair of ideals, for all $i\geq 0$. The main subject of the paper is to study the finiteness properties and Artinianness of $\mathfrak{F}^i_{\mathfrak{a},I,J}(M)$ and $\check{\mathfrak{F}}^i_{\mathfrak{a},\mathfrak{m},J}(M)$. We  study the maximum and minimum integer $i\in \N$ such that $\mathfrak{F}^i_{\mathfrak{a},\mathfrak{m},J}(M)$ and $\check{\mathfrak{F}}^i_{\mathfrak{a},\mathfrak{m},J}(M)$ are not Artinian.  We obtain some results involving cossuport, coassociated and attached primes for formal local cohomology modules with respect to a pair of ideals. Also, we  give an criterion involving the concepts of finiteness and vanishing of formal local cohomology modules and \v{C}ech-formal local cohomology modules with respect to a pair of ideals.
\end{abstract}
\newpage
\maketitle
\section{Introduction}
\hspace{0.5cm}Throughout this paper $R$ is a commutative Noetherian  ring (non-zero identity),  $\mathfrak{a},\mathfrak{b},I$ and $J$ are ideals of $R$, and $M$ be a non-zero finitely generated $R$-module. For $i \in \mathbb{N}$, $H^i_{\mathfrak{a}}(M)$ denote the $i$th  local cohomology module of $M$ with respect to $\mathfrak{a}$ (see \cite{grot}, \cite{Hart}, \cite{24h}). This concept is an important tool in algebraic geometry and commutative algebra, and has been studied by several authors.

When $(R, \mathfrak{m})$ is a local ring, Schenzel \cite{art0}  defined an object of study as follows. Let $\underline{x}= x_1,\ldots,x_r$  be a system of elements of $R$ with $\mathfrak{b}={\rm Rad}(\underline{x}R)$, and $\check{C}_{\underline{x}}$ denote the \v{C}ech complex of $R$ with respect to $\underline{x}$. The projective system of $R$-modules $\{M / \mathfrak{a}^n M  \}_{n\in \mathbb{N}}$ induces a projective system of $R$-complexes $\{\check{C}_{\underline{x}}\otimes  M / \mathfrak{a}^n M  \}$.
Consider the projective limit $\displaystyle \varprojlim(\check{C}_{\underline{x}}\otimes  M / \mathfrak{a}^n M)$.

For an integer $i\in \mathbb{Z}$, the cohomology module $H^{i}(\displaystyle \varprojlim(\check{C}_{\underline{x}}\otimes  M / \mathfrak{a}^n M))$ is called the $i$th $\mathfrak{a}$-formal local cohomology with respect to $\mathfrak{b}$, denoted by $\check{\mathfrak{F}}^i_{\mathfrak{a,b}}(M)$. In the case of $\mathfrak{b}= \mathfrak{m}$, we speak simply  the $i$th $\mathfrak{a}$-formal local cohomology.

Now, consider the family of local cohomology modules $\{H^i_{\mathfrak{b}}(M / \mathfrak{a}^n M)\}_{n\in \mathbb{N}} $. For every integer $n$, there is a natural homomorphism $ H^i_{\mathfrak{b}}(M / \mathfrak{a}^{n+1}M)\rightarrow H^i_{\mathfrak{b}}(M / \mathfrak{a}^n M)$ such that the family forms a projective system.
 Let the projective limit, $\mathfrak{F}^i_{\mathfrak{a,b}}(M):=\displaystyle \varprojlim H^i_{\mathfrak{b}}(M / \mathfrak{a}^n M)$. When $\mathfrak{b}=\mathfrak{m}$, Schenzel \cite{art0} has shown the following isomorphism $ \mathfrak{\check{F}}^i_{\mathfrak{a},\mathfrak{m}}(M)\cong \mathfrak{F}^i_{\mathfrak{a},\mathfrak{m}}(M)$, showing  the relation between formal local cohomology and projective limits of some local cohomology modules. An approach, but not least, was studied by  Peskine and Szpiro \cite[Chapter III]{18} and
Faltings \cite{faltings}.

Through the concept introduced in \cite{art1} of local cohomology defined by a pair of ideals, the authors \cite{paper} introduced the  notion of formal local cohomology defined by  a pair of ideals.  More explicity, for an integer $i\in \mathbb{Z}$, the cohomology module $H^i(\displaystyle \varprojlim(\check{C}_{\underline{x}, J}\otimes  M / \mathfrak{a}^n M))$ called the $i$th \v{C}ech  $\mathfrak{a}$-formal cohomology with respect to $(I,J)$, denoted by $\check{\mathfrak{F}}^i_{\mathfrak{a}, I, J}(M)$. After this, considerthe projective limit  $\displaystyle \varprojlim H^i_{I,J}(M / \mathfrak{a}^n M)$,  denoted by $\mathfrak{F}^i_{\mathfrak{a},I,J}(M)$.

Note that if $J=0$, $\check{C}_{\underline{x}, J}$ coincides with the ordinary \v{C}ech complex $\check{C}_{\underline{x}}$ with respect to $\underline{x}= x_1,\ldots,x_s$. Therefore $ \check{\mathfrak{F}}^i_{\mathfrak{a},I,0}(M)\cong  \check{\mathfrak{F}}^i_{\mathfrak{a},I}(M)$. If $J=0$ and $I=\mathfrak{m}$,  we have $ \check{\mathfrak{F}}^i_{\mathfrak{a},\mathfrak{m},0}(M)\cong  \check{\mathfrak{F}}^i_{\mathfrak{a},\mathfrak{m}}(M)$. The same conclusion we have for $\mathfrak{F}^i_{\mathfrak{a},I,J}(M)$ when $J=0$. \noindent

These new definitions are a natural generalization of $\mathfrak{a}$-formal local cohomology  with respect to $\mathfrak{b}$ and $\mathfrak{a}$-formal local cohomology, both introduced by Schenzel \cite{art0} and discussed by Mafi  \cite{131}, Asgharzadeh and Divaaani-Aazar \cite{finiteness}, Gu \cite{yan}, Bijan-Zadeh and Rezaei \cite{ffdepth1}, and Eghbali \cite{majidp}. However the isomorphism between $\mathfrak{\check{F}}^i_{\mathfrak{a},I,J}(M)$ and $\mathfrak{F}^i_{\mathfrak{a},I,J}(M)$ does not happen always, differently as the ordinary formal local cohomology.

About artinianness of the ordinary local cohomology, Huneke \cite {huneke} ask when $H^i_I(M)$ is artinian for an integer $i$. In general this question is not true (see \cite {14} and \cite{10}). In this way, some results have been achieved such as, (1) if ${\rm dim}M=d$, $H^d_I(M)$ is artinian for any finitely generated $R$-module $M$; (2) the characterization of the least integer $s$ such that $H^s_I(M)$ is not artinian, (3) $H^i_{\mathfrak{m}}(M)$ is always artinian when $(R,\mathfrak{m})$ is a local ring and $M$ finitely generated.
For local cohomology defined by a pair of ideals, Chu and Wang \cite {art3} generalized (1) and (2) and gave  other results in this context. About the generalization  of (3) in this sense, Tehranian and Talemi \cite[Theorem 2.10]{nonart} have shown that $H^i_{\mathfrak{m},J}(M)$ is not artinian for some $i\in \mathbb{N}_0$, when $J$ is a non-nilpotent ideal of a local ring $(R,\mathfrak{m})$ and $M$ a finite $R$-module. We cite \cite{parsa} for more interesting results involving artinianness of local cohomology defined by a pair of ideals.

The artinianness of formal local cohomology modules has been studied by  Gu \cite{yan}, Asgharzadeh-Divaani-Aazar \cite{art3}, Mafi \cite{131}, Bijan-Zadeh and Rezaei \cite{ffdepth1}, and Eghbali \cite{majidp}. In this context was shown that, among other results, (1) For an local ring $(R,\mathfrak{m})$ and $M$ finitely generated $R$-module of dimension $d$, $\mathfrak{\check{F}}^d_{\mathfrak{a},\mathfrak{m}}(M)$ is artinian; (2) $\mathfrak{\check{F}}^i_{\mathfrak{a},\mathfrak{m}}(M)$ is artinian for all $i<t$ if, and only if, $\mathfrak{a}\subseteq {\rm Rad}(0: \mathfrak{\check{F}}^i_{\mathfrak{a},\mathfrak{m}}(M))$ for all $i<t$; (3) If, for some integer $t$, $\check{\mathfrak{F}}^i_{\mathfrak{a}, \mathfrak{m}}(M)$ is artinian for all $i<t$, then $\check{\mathfrak{F}}^t_{\mathfrak{a}, \mathfrak{m}}(M)/ \mathfrak{a}\check{\mathfrak{F}}^t_{\mathfrak{a}, \mathfrak{m}}(M)$ is artinian; (4) some relations about ${\rm inf}\{ i \  | \  \mathfrak{F}^i_{\mathfrak{a},\mathfrak{m}}(M) \  {\rm is \  not \  artinian} \}$ .
The purpose of this paper is, using the concept of formal local cohomology with respect to a pair of ideals introduced in \cite{paper}, generalize (1), (2), (3) and give some results involving the ${\rm inf}\{ i \  | \  \mathfrak{F}^i_{\mathfrak{a},\mathfrak{m},J}(M) \  {\rm is \  not \  artinian} \}$.
Moreover, we  give a characterization of prime attached of the $\mathfrak{F}^i_{\mathfrak{a},I,J}(M)$ and some results involving the concept of cosupport and coassociated primes.

The organization of the paper is as follows. In the next section, we will set the formal filter depth and formal \v{C}ech-depth and study some elementary properties about this new concept.

After discussing basic properties about this new concept, in the Section 3 we discuss and give a clue about the non-artinianness for formal local cohomology modules with respect to a pair of ideals.

In the Section 4, we investigate the question of artinianness of formal local cohomology and \v{C}ech-formal local cohomology modules with respect to a pair of ideals. The main result of this section generalizes (1) previously cited.

In Section 5, we study some results involving the concept of cosupport and coassociated primes. Also, we give a characterization of coassociated primes of top formal local cohomology modules with respect to a pair of ideals.

In the last section, we obtain some results on finiteness of $\mathfrak{F}^i_{\mathfrak{a},\mathfrak{m},J}(M)$ and $\mathfrak{\check{F}}^i_{\mathfrak{a},\mathfrak{m},J}(M)$.  We  give an important criterion involving the concepts of finiteness and vanishing of formal local cohomology modules and \v{C}ech-formal local cohomology modules with respect to a pair of ideals.

For the other sections, assume always that $R$ is a local ring with maximal ideal $\mathfrak{m}$, denoted by $(R,\mathfrak{m})$.

\section{The Formal Filter Depth  of Formal Local Cohomology Modules }

\hspace{0.5cm} Let $M$ be a finitely generated $R$-module. Remember that a sequence $x_1,\dots,x_n$ of elements in $R$ is said to be an $M$-filter regular
sequence if,  for all  ${\mathfrak p}\in {\rm Supp}(M)\setminus{\mathfrak m}$, the sequence
$x_1/1,\dots,x_n/1$ of elements of $R_{\mathfrak p}$
is an $M$- filter regular sequence $[9]$. For an ideal
$I$ of $R$, the $f$-depth of $I$ on
$M$ is defined as the length of any maximal
$M$-filter regular sequence in $I$, denoted by
$f$-depth$(I, M)$. When a maximal
$M$-filter regular sequence in $I$ does not
exist, we understand that the length is $\infty$.

\begin{prop}\label{infimo} Let $M$ be a finite $R$-module. Let $t$ be a non-negative integer such that $\mathfrak{F}^i_{\mathfrak{a},\mathfrak{m},J}(M)$ is Artinian for all $i<t$. Then $\mathfrak{F}^i_{\mathfrak{a},\mathfrak{m},K}(M)$ is Artinian for all $K\subset J$ and all $i<t$. In particular $${\rm inf}\{ i \  | \  \mathfrak{F}^i_{\mathfrak{a},\mathfrak{m},J}(M) \  {\rm is \  not \  artinian} \}\leq f\!-\!{\rm depth}(I,M).$$
\end{prop}
\begin{proof} We will use induction on $t$. If $t=0$, there is nothing to prove. Let $t=1$. Since $\Gamma_{{\mathfrak m},K}(M/{\mathfrak a}^nM)\subset \Gamma_{{\mathfrak m},J}(M/{\mathfrak a}^nM)$ \cite[Proposition 1.4 (3)]{art1}, we have $\displaystyle \lim_{ ^{\longleftarrow} }\Gamma_{{\mathfrak m},K}(M/{\mathfrak a}^nM)\subset \displaystyle \lim_{ ^{\longleftarrow} }\Gamma_{{\mathfrak m},J}(M/{\mathfrak a}^nM)$  for all $K\subset J,$ the assertion holds.

Suppose $t>1$. Since $H_{{\mathfrak m},J}^i(M/{\mathfrak a}^nM)\cong H_{{\mathfrak m},J}^i\left(\frac{M/{\mathfrak a}^nM}{\Gamma_{{\mathfrak m},J}(M/{\mathfrak a}^nM)}\right)$ for all $i>0$ \cite[Corollary 1.13 (4)]{art1}, so we have

$$\displaystyle \varprojlim H_{{\mathfrak m},J}^i(M/{\mathfrak a}^nM)\cong \displaystyle \varprojlim H_{{\mathfrak m},J}^i\left(\frac{M/{\mathfrak a}^nM}{\Gamma_{{\mathfrak m},J}(M/{\mathfrak a}^nM)}\right)$$
\noindent for all $i\geq 0$. Thus we may assume that $\Gamma_{{\mathfrak m},J}(M/{\mathfrak a}^nM)=0$. Let $E$ be an injective hull of $\overline{M}:=M/{\mathfrak a}^nM$ for all $n$,  and  $N_n=E/\overline{M}$. Then, follows that
$H_{{\mathfrak m},J}^{i+1}(M/{\mathfrak a}^nM)\cong H_{{\mathfrak m},J}^i\left(N_n\right)$ and $H_{{\mathfrak m},K}^{i+1}(M/{\mathfrak a}^nM)\cong H_{{\mathfrak m},K}^i\left(N_n\right)$ for all $i\geq 0$. So
$\mathfrak{F}^{i+1}_{\mathfrak{a},\mathfrak{m},J}(M)\cong \displaystyle \varprojlim H_{{\mathfrak m},J}^i\left(N_n\right)$ and $\mathfrak{F}^{i+1}_{\mathfrak{a},\mathfrak{m},K}(M)\cong \displaystyle \varprojlim H_{{\mathfrak m},K}^i\left(N_n\right)$ for all $i\geq 0$. Thus $N_n$ satisfies our induction hypothesis. Therefore, $\mathfrak{F}^i_{\mathfrak{a},\mathfrak{m},K}(M)$
 is Artinian for all $K\subset J$ and all $i<t.$
 If $K=0$, so the last statement follows by \cite[Theorem 3.1]{16}.
 \end{proof}

The previous result motivates to consider the following definition.

\begin{defn} \label{definition} Let $\mathfrak{a}$ an ideal of the ring $(R,\mathfrak{m})$. The least integer such that formal local cohomology and \v{C}ech-formal local cohomology are not artinian are called respectively by the formal filter depth and formal \v{C}ech- filter depth with respect to a pair of ideals $(\mathfrak{m},J)$. More specifically

$$ {\rm ff}\!-\!{\rm depth}(\mathfrak{a},\mathfrak{m},J,M):= {\rm inf}\{ i \  | \  \mathfrak{F}^i_{\mathfrak{a},\mathfrak{m},J}(M) \  {\rm is \  not \  artinian} \},$$
$${\rm f{\check f}}\!-\!{\rm depth}(\mathfrak{a},\mathfrak{m},J,M):= {\rm inf}\{ i \  | \  \check{\mathfrak{F}}^i_{\mathfrak{a},\mathfrak{m},J}(M) \  {\rm is \  not \  artinian} \}.$$

\noindent Analogously, we define the greatest integer such that de formal local cohomology and \v{C}ech-formal local cohomology are not artinian by

$$ {\rm gg}\!-\!{\rm depth}(\mathfrak{a},\mathfrak{m},J,M):= {\rm sup}\{ i \  | \  \mathfrak{F}^i_{\mathfrak{a},\mathfrak{m},J}(M) \  {\rm is \  not \  artinian} \},$$
$$ {\rm g{\check g}}\!-\!{\rm depth}(\mathfrak{a},\mathfrak{m},J,M):= {\rm sup}\{ i \  | \  \check{\mathfrak{F}}^i_{\mathfrak{a},\mathfrak{m},J}(M) \  {\rm is \  not \  artinian} \}.$$

\end{defn}

This definitions are  natural generalizations of formal filter depth given in \cite[Definition 2.1]{yan} and \cite[Definition 2.11]{ffdepth1}.
By \cite[Proposition 2.4]{paper} is easy to see that
$${\rm f{\check f}}\!-\!{\rm depth}(\mathfrak{a},\mathfrak{m},J,M)\leq {\rm ff}\!-\!{\rm depth}(\mathfrak{a},\mathfrak{m},J,M),$$
$$  {\rm gg}\!-\!{\rm depth}(\mathfrak{a},\mathfrak{m},J,M)\leq  {\rm g{\check g}}\!-\!{\rm depth}(\mathfrak{a},\mathfrak{m},J,M).$$

The next result shows that the formal filter depth is invariant by radicals.
\begin{prop}\label{inv} Let $\mathfrak{a}$ and $\mathfrak{b}$ be two ideals of $(R,\mathfrak{m})$ such that $Rad(\mathfrak{a})= Rad(\mathfrak{b}).$ Then we have that ${\rm ff}\!-\!{\rm depth}(\mathfrak{a},\mathfrak{m},J,M)= {\rm ff}\!-\!{\rm depth}(\mathfrak{b},\mathfrak{m},J,M)$.
\end{prop}
\begin{proof} By \cite[Theorem 6.7]{paper} we have that
$$\mathfrak{F}^i_{\mathfrak{a},\mathfrak{m},J}(M)\cong {\rm Hom}_R(H^{-i}_{\mathfrak{a}}({\rm Hom}_R(M,D_{\underline{x},J})), E_R(\mathbb{K})) \ \  {\rm and}$$
$$\mathfrak{F}^i_{\mathfrak{b},\mathfrak{m},J}(M)\cong {\rm Hom}_R(H^{-i}_{\mathfrak{b}}({\rm Hom}_R(M,D_{\underline{x},J})), E_R(\mathbb{K})).$$
Since the usual local cohomology is invariant by radicals, we have the statement.
\end{proof}

\begin{prop} ${\rm ff\!-\!depth}(\mathfrak{a},\mathfrak{m},J,M)= {\rm ff}\!-\!{\rm depth}(\mathfrak{a}\widehat{R},\mathfrak{m}\widehat{R},J\widehat{R},\widehat{M})$.
\end{prop}
\begin{proof}The statement follows because $\mathfrak{F}^i_{\mathfrak{a},\mathfrak{m},J}(M)\cong \mathfrak{F}^i_{\mathfrak{a}\widehat{R},\mathfrak{m}\widehat{R},J\widehat{R}}(\widehat{M}) $ by \cite[Theorem 2.3]{paper}.
\end{proof}
The next result extends \cite[Proposition 2.4]{yan}

\begin{prop} Let $\mathfrak{a}\subseteq \mathfrak{b}$ be two ideals of $(R,\mathfrak{m})$. We have the following inequality
$${\rm ff\!-\!depth}(\mathfrak{a},\mathfrak{m},J,M)\leq {\rm ff\!-\!depth}(\mathfrak{b},\mathfrak{m},J,M)+ {\rm ara}(\mathfrak{b}/\mathfrak{a}).$$
\end{prop}
\begin{proof} By Proposition \ref{inv}, we may assume that exists $x_1,\ldots,x_n \in R$ such that $\mathfrak{b}=\mathfrak{a}+(x_1,\ldots,x_n)$. We proceed by induction on $n$. However, it is sufficient to show only the case $n=1$. By \cite[Theorem 6.11]{paper}, there is the long exact sequence
\begin{equation}\label{equation 1}
\cdots\rightarrow \mathfrak{F}^i_{\mathfrak{a},\mathfrak{m},J}(M)\rightarrow \mathfrak{F}^i_{\mathfrak{b},\mathfrak{m},J}(M)\rightarrow {\rm Hom}_R(R_{x,J},\mathfrak{F}^{i+1}_{\mathfrak{a},\mathfrak{m},J}(M))\rightarrow\cdots.
\end{equation}

\noindent For all $i< {\rm ff\!-\!depth}(\mathfrak{a},\mathfrak{m},J,M)-1$,  $ \mathfrak{F}^i_{\mathfrak{a},\mathfrak{m},J}(M)$ and $\mathfrak{F}^{i+1}_{\mathfrak{a},\mathfrak{m},J}(M)$ are artinian by Definition \ref{definition}. Then,  ${\rm Hom}_R(R_{x,J},\mathfrak{F}^{i+1}_{\mathfrak{a},\mathfrak{m},J}(M))$ is artinian  by the exact sequence \ref{equation 1}, and therefore ${\rm ff\!-\!depth}(\mathfrak{a},\mathfrak{m},J,M)\leq {\rm ff\!-\!depth}(\mathfrak{b},\mathfrak{m},J,M)+1.$
\end{proof}

Our next result shows the relation between ${\rm f\!-\!depth}M$  and the greatest integer such that the formal local cohomology defined by a pair of ideals is non-zero. This result improves \cite[Proposition 2.7 (1)]{yan}.

\begin{prop}Let $\mathfrak{a}, J$ be two ideals of $(R,\mathfrak{m})$, and let  ${\rm f\!-\!depth}M <\infty$. Then $${\rm ff\!-\!depth}(\mathfrak{a},\mathfrak{m},J,M)\leq {\rm min}\{ {\rm f\!-\!depth}(I,M), {\rm dim} M/(\mathfrak{a}+J)M\}.$$
\end{prop}

\begin{proof}
 Firstly note that
$$
\begin{array}{lll}
{\rm ff\!-\!depth}(\mathfrak{a},\mathfrak{m},J,M)&\leq& {\rm sup}\{ i \  | \  \mathfrak{F}^i_{\mathfrak{a},\mathfrak{m},J}(M) \  {\rm is \  not \  artinian} \}\\
&\leq& {\rm sup}\{ i \  | \  \mathfrak{F}^i_{\mathfrak{a},\mathfrak{m},J}(M)\neq 0 \}= {\rm dim}M/(\mathfrak{a}+J)M.  \end{array}
$$

The last equality follows by \cite[Theorem 4.2]{paper}. Now, by Proposition \ref{infimo} we have that ${\rm ff\!-\!depth}(\mathfrak{a},\mathfrak{m},J,M) \leq {\rm inf}\{ i \  | \  H^i_{\mathfrak{m}}(M) \  {\rm is \  not \  artinian} \}$.

\end{proof}

\section{Non-Artinian of $ \mathfrak{F}^i_{\mathfrak{a},\mathfrak{m},J}(M)$ }\label{section3}

\hspace{0.5cm} A natural question that arises when one is dealing on artinianness of formal local cohomology  is to know if, for an  finite $R$-module, the formal local cohomology is not Artinian for some integer $i$. The
objective of this section is give some answer about this question.

For this purpose, we give an important and essential construction for our next result. In \cite[Theorem 5.1]{art1}, Takahashi and {\it etal} showed a generalization of the local duality theorem. More explicit, let $(R,\mathfrak{m})$ be a Cohen Macaulay complete local ring of dimension $d$ and $J$ be a perfect ideal of $R$, i.e, ${\rm grade}(J,R)= {\rm pd}_R R/J=t$. Then, if $M$ is a finitely generated $R$-module, then there is the isomorphism
$$H_{\mathfrak{m},J}^i(M)\cong {\rm Ext}^{d-t-i}_R(M,S)^\vee,$$ for all integer $i$ and $S= H^{d-t}_{\mathfrak{m},J}(R)^\vee$.

With this, in the context of formal local cohomology defined by a pair of ideals we have that
$$ \mathfrak{F}^i_{\mathfrak{a},\mathfrak{m},J}(M)\cong {\rm Hom}_R(\displaystyle \varinjlim{\rm Ext}_R^{d-t-i}(M/\mathfrak{a}^nM,S),E_R(\mathbb{K})).$$
Note that, for all $i\in \mathbb{Z}$, $\displaystyle \varinjlim{\rm Ext}_R^{d-t-i}(M/\mathfrak{a}^nM,S)$ is exactly the generalized local cohomology with respect to $\mathfrak{a}$  (denoted by $H^{d-t-i}_{\mathfrak{a}}(M,S)$), introduced by Herzog  \cite{herzog}. Therefore
$$ \mathfrak{F}^i_{\mathfrak{a},\mathfrak{m},J}(M)\cong H^{d-t-i}_{\mathfrak{a}}(M,S)^\vee$$ where $(-)^\vee= {\rm Hom}_R(-,E_R(\mathbb{K}))$, $i \in \mathbb{Z}$. This shows the relation between the formal local cohomology defined by a pair of ideals and the Matlis' dual of certain generalized local cohomology with respect to $\mathfrak{a}$.
Now we are able to show the next result, that generalizes \cite[Theorem 2.16]{yan}.

\begin{thm} Let $\mathfrak{a}, \mathfrak{b}$ and $J$ be ideals of $(R,\mathfrak{m})$ such that $\mathfrak{a}\subseteq\mathfrak{b}$, and $J$ a perfect ideal of $R$ of grade $t$. If $M$ is a finitely generated $R$ module of dimension $d$, then there is a surjective homomorphism $\mathfrak{F}^{d-t}_{\mathfrak{a},\mathfrak{m},J}(M)\rightarrow \mathfrak{F}^{d-t}_{\mathfrak{b},\mathfrak{m},J}(M)$. In particular $\mathfrak{F}^{d-t}_{\mathfrak{b},\mathfrak{m},J}(M)$ is a quocient of $H^{d-t}_{\mathfrak{m},J}(M)$.

\end{thm}
\begin{proof}Firstly, let $\overline{R}=R/{\rm Ann}_R M$. By \cite[Theorem 2.3]{paper}, we have that $ \mathfrak{F}^i_{\mathfrak{a},\mathfrak{m},J}(M)\cong \mathfrak{F}^i_{\mathfrak{a}\overline{R},\mathfrak{m}\overline{R},J\overline{R}}(M)$ and $ \mathfrak{F}^i_{\mathfrak{b},\mathfrak{m},J}(M)\cong \mathfrak{F}^i_{\mathfrak{b}\overline{R},\mathfrak{m}\overline{R},J\overline{R}}(M)$. Thus, we can assume that ${\rm Ann}_R M=0$ and then ${\rm dim}R=d$. Now, since $\mathfrak{F}^i_{\mathfrak{a},\mathfrak{m},J}(M)\cong \mathfrak{F}^i_{\mathfrak{a}\widehat{R},\mathfrak{m}\widehat{R},J\widehat{R}}(\widehat{M})$ \cite[Theorem 2.3]{paper} we can consider that $R$ is complete.

By Cohen's Structure Theorem, there exists a complete regular local ring $(T,\mathfrak{n})$ such that $R\cong T/I$ for some ideal $I$ of $T$. Set $\mathfrak{a}_1=\mathfrak{a}\cap I$ and $\mathfrak{b}_1=\mathfrak{b}\cap I$. By \cite[Theorem 2.3]{paper} we have that $\mathfrak{F}^i_{\mathfrak{a},\mathfrak{m},J}(M)\cong \mathfrak{F}^i_{\mathfrak{a}_1,\mathfrak{m},J}(M)$ and $\mathfrak{F}^i_{\mathfrak{b},\mathfrak{m},J}(M)\cong \mathfrak{F}^i_{\mathfrak{b}_1,\mathfrak{m},J}(M)$ for all $i\geq 0$. Since ${\rm dim}_R M={\rm dim}_T M$, we may assume that $R=T$.
Using the previous comment we have that
$$ \mathfrak{F}^{d-t}_{\mathfrak{a},\mathfrak{m},J}(M)\cong  \mathfrak{F}^{d-t}_{\mathfrak{a}_1,\mathfrak{m},J}(M)\cong {\rm Hom}_T(H^{0}_{\mathfrak{a}_1}(M,S),E_T(T/ \mathfrak{n}))  \ \ {\rm and} $$
$$ \mathfrak{F}^{d-t}_{\mathfrak{b},\mathfrak{m},J}(M)\cong  \mathfrak{F}^{d-t}_{\mathfrak{b}_1,\mathfrak{m},J}(M)\cong {\rm Hom}_T(H^{0}_{\mathfrak{b}_1}(M,S),E_T(T/ \mathfrak{n})). $$

This finishes the proof because $H^{0}_{\mathfrak{b}_1}(M,S)$ is a submodule of $H^{0}_{\mathfrak{a}_1}(M,S)$.
\end{proof}

\begin{rem}By previous result follows that $\mathfrak{F}^{d-t}_{\mathfrak{b},\mathfrak{m},J}(M)$ is Artinian if $H^{d-t}_{\mathfrak {m},J}(M)$ is Artinian. The proof of the \cite[Theorem 5.1]{art1} showed that  $H^{i}_{\mathfrak {m},J}(M)=0$ for all $i\neq d-t$. Moreover, if $J$ is non-nilpotent ideal of $R$,  $H^{d-t}_{\mathfrak {m},J}(M)$ is not Artinian \cite[Theorem 2.10]{nonart}.
Therefore, if $J$ is non-nilpotent ideal and with the same hypotheses of  Theorem 2.6, $\mathfrak{F}^{d-t}_{\mathfrak{b},\mathfrak{m},J}(M)$ is not Artinian if the kernel of the surjective map above is artinian.

%***isso pode ser util para saber se FLC é ou nao artiniana em d-t???

\end{rem}

\begin{rem}\label{rem3}  With the same hypotheses of \cite[Theorem 2.9, (a)]{nonart}, is possible to show other version of Local Duality theorem generalized.  By previous comment the Theorem 2.6, more explicit, if $(R,\mathfrak{m})$ be a local ring, $M$ be a finite Cohen-Macaulay $R$-module such that ${\rm dim}M={\rm dim}R=d$ and $J$ be an ideal generated by a regular sequence $x_1,\ldots,x_t$ in $R$ we have that
$$ \mathfrak{F}^i_{\mathfrak{a},\mathfrak{m},J}(M)\cong H^{d-t-i}_{\mathfrak{a}}(M,S)^\vee$$ where $(-)^\vee= {\rm Hom}_R(-,E_R(\mathbb{K}))$, $i \in \mathbb{Z}$.
Therefore we have other version of the Theorem 2.6 for $R$, i.e, there is a surjective homomorphism $\mathfrak{F}^{d-t}_{\mathfrak{a},\mathfrak{m},J}(R)\rightarrow \mathfrak{F}^{d-t}_{\mathfrak{b},\mathfrak{m},J}(R)$.
With this, is possible to obtain the same of previous remark.
%*** como podemos fazer isso para $M$? é possivel usar o Theorem 2.9 (b) do Non artinian para obter outra versao do Local Duality Theorem??
\end{rem}

\section{Artinianness of Formal Local Cohomology Modules}

\hspace{0.5cm} One of the important problems in local cohomology is to investigate artinianness properties. The main purpose of this section is give some results about the artinianness of formal local cohomology and \v{C}ech-formal local cohomology modules with respect to a pair of ideals.

Let $t$ be an integer. It is well known that the local cohomology module $H_I^i(M)$ is finitely generated for all $i<t$ if, and only if, there is some integer $r>0$ such that  $I^rH_I^i(M)=0$ for all $i<t$. Similar results were obtained by Yan Gu \cite{yan}, Bijan-Zadeh and Rezaei \cite{ffdepth1}, and  Mafi \cite{13} in the context of ordinary formal local cohomology. More specifically, $\mathfrak{\check{F}}^i_{I}(M)$ is Artinian for all $i<t$ (respectively for $i>t$) if and only if there is some integer $r>0$ such that  $I^r\mathfrak{\check{F}}^i_{I}(M)=0$ for all $i<t$ (respectively for $i>t$). The following theorem extends this previous result commented.

\begin{thm}  Let $\mathfrak{a},J$ ideals of  $(R, \mathfrak{m})$ and $M$ a finitely generated $R$-module. Let $t$ be a non-negative integer. If $\Gamma_{\mathfrak{a}}(M)$ is an $J$-torsion $R$-module the following statements are equivalent:
\begin{itemize}
\item[(a)] $\mathfrak{\check{F}}^i_{\mathfrak{a},\mathfrak{m},J}(M)$ is Artinian for all $i>t$;
\item[(b)] $\mathfrak{a}\subseteq {\rm Rad}(0: \mathfrak{\check{F}}^i_{\mathfrak{a},\mathfrak{m},J}(M))$ for all $i>t$.
\end{itemize}
\end{thm}

\begin{proof} $(a)\Rightarrow (b)$ Let $i>t$. Since $\mathfrak{\check{F}}^i_{\mathfrak{a},\mathfrak{m},J}(M)$ is Artinian and $R$ local ring, we have that $\mathfrak{a}^j\mathfrak{F}^i_{\mathfrak{a},\mathfrak{m},J}(M)=0$  for some positive integer $j$.  Then $\mathfrak{a}\subseteq {\rm Rad}(0: \mathfrak{\check{F}}^i_{\mathfrak{a},\mathfrak{m},J}(M))$ for all $i>t$.

$(b)\Rightarrow (a)$ We prove it by induction on ${\rm dim}M= d$. If $d=0$, we have  $\mathfrak{\check{F}}^i_{\mathfrak{a},\mathfrak{m},J}(M))=\mathfrak{F}^i_{\mathfrak{a},\mathfrak{m},J}(M))=0$ for all $i>0$ \cite[Proposition 4.1]{paper}. Therefore the result is clearly true.
Next, we assume that $d>0$ and that the claim is true for all values less than $d$. Firstly, since $\Gamma_{\mathfrak{a}}(M)$ is ${\mathfrak a}$-torsion and  $J$-torsion we have that $\mathfrak{\check{F}}^i_{\mathfrak{a},\mathfrak{m},J}(\Gamma_{\mathfrak{a}}(M))\cong H^i_{\mathfrak{m}}(\Gamma_{\mathfrak{a}}(M))$. Then from the exact sequence
$$0\rightarrow  \Gamma_{\mathfrak{a}}(M)\rightarrow M \rightarrow M/\Gamma_{\mathfrak{a}}(M)\rightarrow 0,$$  we have by \cite[Theorem 3.4]{paper} the long exact sequence
$$\ldots \rightarrow H^i_{\mathfrak{m}}(\Gamma_{\mathfrak{a}}(M))\rightarrow \mathfrak{\check{F}}^i_{\mathfrak{a},\mathfrak{m},J}(M)\rightarrow \mathfrak{\check{F}}^i_{\mathfrak{a},\mathfrak{m},J}(M/\Gamma_{\mathfrak{a}}(M))\rightarrow H^{i+1}_{\mathfrak{m}}(\Gamma_{\mathfrak{a}}(M))\rightarrow \ldots   (\sharp).$$
So, if $\mathfrak{\check{F}}^i_{\mathfrak{a},\mathfrak{m},J}(M/\Gamma_{\mathfrak{a}}(M))$ is artinian for all $i>t$, then we have the statement. From $(\sharp)$, we can see that $\mathfrak{a}\subseteq {\rm Rad}(0: \mathfrak{\check{F}}^i_{\mathfrak{a},\mathfrak{m},J}(M/\Gamma_{\mathfrak{a}}(M)))$ for all $i>t$. Then, we can assume that $\Gamma_{\mathfrak{a}}(M)=0$. Let $x\in \mathfrak{a}$ an $M$-regular element. By hypothesis, for all $i>t$, there is $j_i$ a positive integer  such that $x^{j_i}\mathfrak{\check{F}}^i_{\mathfrak{a},\mathfrak{m},J}(M)=0$. The exact sequence
$$0\rightarrow M\stackrel{x^{j_i}}{\rightarrow} M\rightarrow M/x^{j_i}M\rightarrow 0$$  induces the exact sequence
$$0\rightarrow \mathfrak{\check{F}}^i_{\mathfrak{a},\mathfrak{m},J}(M) \rightarrow \mathfrak{\check{F}}^i_{\mathfrak{a},\mathfrak{m},J}(M/x^{j_i}M)\rightarrow \mathfrak{\check{F}}^{i+1}_{\mathfrak{a},\mathfrak{m},J}(M)$$
for all $i>t$. This sequence provides  that  $\mathfrak{a}\subseteq {\rm Rad}(0: \mathfrak{\check{F}}^i_{\mathfrak{a},\mathfrak{m},J}(M/x^{j_i}M))$, and   $\mathfrak{\check{F}}^i_{\mathfrak{a},\mathfrak{m},J}(M/x^{j_i}M))$ is artinian for all $i>t$ by inductive hypothesis. Therefore $\mathfrak{\check{F}}^i_{\mathfrak{a},\mathfrak{m},J}(M)$ is artinian for all $i>t$.
\end{proof}

%\textcolor{red}{ ESTE ESTA FORA DA SEQUENCIA!!!! Note that $(R,\mathfrak{m})$ be a Cohen Macaulay complete local ring of dimension $d$, $J$ be a perfect ideal of $R$, i.e, ${\rm grade}(J,R)= {\rm pd}_R R/J=t$ and $M$ a finitely generated $R$-module. By Remark \ref{rem3}, we have that the formal local cohomology is the Matlis' dual of certain generalized local cohomology with respect to $\mathfrak{a}$. Then in this case, the previous result is an dual version of generalized finiteness criterion for local cohomology \cite[Proposition 9.1.2]{grot}.}

\begin{cor}  Let $\mathfrak{a},J$ be two ideals of $(R, \mathfrak{m})$. Let $M$ be a finitely generated $R$-module and $t$ be a non-negative integer. If $\Gamma_{\mathfrak{a}}(M)$ is an $J$-torsion $R$-module, then
$$ g{\check g}\!-\!{\rm depth}(\mathfrak{a},\mathfrak{m},J,M)={\rm sup}\{i:\mathfrak{a} \not\subseteq {\rm Rad}(0: \mathfrak{\check{F}}^i_{\mathfrak{a},\mathfrak{m},J}(M))\}.$$

\end{cor}

\begin{rem}\label{remark5.}

\begin{itemize}{\rm
\item[1)] With previous result and \cite[Proposition 2.4]{paper} is possible to show that: If  $\mathfrak{a}\subseteq {\rm Rad}(0: \mathfrak{\check{F}}^i_{\mathfrak{a},\mathfrak{m},J}(M))$ for all $i>t$, then $\mathfrak{F}^i_{\mathfrak{a},\mathfrak{m},J}(M)$ is artinian for all $i>t$.

On the other hand, for formal local cohomology $\mathfrak{F}^i_{\mathfrak{a},\mathfrak{m},J}(M)$ the statement $(a)\Rightarrow (b)$  is true too.

Note that, for all the cases in \cite[Corollary 3.5]{paper} the equivalence between (a) and (b) happen for $\mathfrak{F}^i_{\mathfrak{a},\mathfrak{m},J}(M)$.

\item[2)]  Using the same idea of the previous Theorem, is possible to show  that $\mathfrak{\check{F}}^i_{\mathfrak{a},\mathfrak{m},J}(M)$ is Artinian for all $i<t$ if, and only if, $\mathfrak{a}\subseteq {\rm Rad}(0: \mathfrak{\check{F}}^i_{\mathfrak{a},\mathfrak{m},J}(M))$ for all $i<t$.}
\end{itemize}
\end{rem}

\begin{cor}  Let $\mathfrak{a},J$ be two ideals of  $(R, \mathfrak{m})$. Let $M$ be a finitely generated $R$-module and $t$ be a non-negative integer. If $\Gamma_{\mathfrak{a}}(M)$ is an $J$-torsion $R$-module, then
$$ {\rm f{\check f}\!-\!depth}(\mathfrak{a},\mathfrak{m},J,M)={\rm inf}\{i:\mathfrak{a} \not\subseteq {\rm Rad}(0: \mathfrak{\check{F}}^i_{\mathfrak{a},\mathfrak{m},J}(M))\}.$$

\end{cor}

In the following we give an interesting result that improves \cite[Proposition 2.1]{majidp}, \cite[Theorem 3.1]{ffdepth1} and \cite[Theorem 2.1, Theorem 2.2]{artop}. Before that, remember that a prime ideal ${\mathfrak p}$ of $R$ is an {\it atteched prime ideal of} $M$ if, for every finitely genereated ideal $I\subseteq {\mathfrak p}$, there exists $x\in M$ such that $I\subseteq (0:_Rx)\subset {\mathfrak p}$.  Denote  by ${\rm Att}_R(M)$ the set of attached prime ideals of $R$-module $M$.

\begin{thm}\label{top}
Assume $\mathfrak{a}$ be an ideal of  $(R, \mathfrak{m})$. Let $M$ be a finitely generated $R$-module with $\rm{dim} M=d $. Then $\mathfrak{F}^d_{\mathfrak{a},I,J}(M)$ is Artinian for all ideals $I,J$ of $R$. Furthermore
$${\rm Att}_R(\mathfrak{F}^d_{\mathfrak{a},I,J}(M))= \{\mathfrak{p}\in {\rm Supp}_R M\cap V(J) \mid {\rm cd}(I,R/\mathfrak{p})=d\}\cap V(\mathfrak{a}),$$
where ${\rm cd}(I,R/\mathfrak{p})$ is the cohomological dimension of the $R$-module $R/\mathfrak{p}$ with respect to $I$.

\noindent In particular ${\rm Att}_R(\mathfrak{F}^d_{\mathfrak{a},I,J}(M))\subset {\rm Att}_R(\mathfrak{F}^d_{\mathfrak{a},I}(M)).$

\end{thm}
\begin{proof} Let $\overline{R}=R/{\rm Ann}_R M$. Since $H^d_{I,J}(M)\cong H^d_{I\overline{R},J\overline{R}}(M)$ \cite[Lemma 2.1]{artop},  we can consider that $\rm{Ann}M=0$ and so $d=\rm{dim} R$.
Since $H^i_{I,J}(R)=0$ for all $i>d$  \cite[Theorem 4.7]{art1}, and by \cite[Lemma 4.8]{art1} follows that

$$H^d_{I,J}(M/\mathfrak{a}^nM)\cong H^d_{I,J}(R)\otimes_R M/\mathfrak{a}^nM\cong H^d_{I,J}(M)\otimes_R R/\mathfrak{a}^n$$
$$\cong H^d_{I,J}(M)/\mathfrak{a}^n H^d_{I,J}(M).$$

By \cite[Theorem 2.1]{art3}, $H^d_{I,J}(M)$ is Artinian $R$-module. Then, there exist an integer $n_0$ such that for all integer $t\geq n_0$; we have
$ \mathfrak{a}^t H^d_{I,J}(M)= \mathfrak{a}^{n_0} H^d_{I,J}(M).$ Therefore, since
$$\mathfrak{F}^d_{\mathfrak{a},I,J}(M)\cong H^d_{I,J}(M)/\mathfrak{a}^{n_0} H^d_{I,J}(M)$$ follows that is an Artinian $R$-module.
For the second claim use \cite[Proposition 5.2]{melschen}, \cite[Theorem 2.1]{artop} and  the previous isomorphism.

\end{proof}

\begin{thm}\label{top1}
Assume that $\mathfrak{a} $ and $J$ are ideals of  $(R, \mathfrak{m})$. Let $M$ be a finitely generated $R$-module with $\rm{dim} M=d $. Then
$${\rm Att}_R(\mathfrak{F}^d_{\mathfrak{a},\mathfrak{m},J}(M))= \{\mathfrak{p}\in {\rm Supp}_R M\cap V(J) \mid \dim R/{\mathfrak p}=d\}\cap V(\mathfrak{a}).$$
\end{thm}
\begin{proof} The proof follows by \cite[Theorem 2.2]{artop} and Theorem \ref{top}.

\end{proof}

Now an immediate consequence of the Theorem \ref{top}.
\begin{cor}\label{mesmosuporte}Let $\mathfrak{a}, I, J$ be ideals of  $(R, \mathfrak{m})$. Let $M$ and $N$ be two finitely generated $R$-modules of dimension $d$ such that ${\rm Supp}_R M= {\rm Supp}_R N$. Then ${\rm Att}_R(\mathfrak{F}^d_{\mathfrak{a},I,J}(M))= {\rm Att}_R(\mathfrak{F}^d_{\mathfrak{a},I,J}(N))$.

\end{cor}

 \noindent In some results, we will use the following elementary lemma.

\begin{lem}\label{lemma0} Let $R$ be a commutative ring, and let $M$ and $N$ be $R$-modules.
Then ${\rm Ann}_R(M) \cup {\rm Ann}_R(N)\subset {\rm Ann}_R{\rm Ext}^i_R(M,N)$ for all $i$.
\end{lem}
\begin{proof}
Let $x \in {\rm Ann}_R(M) \cup {\rm Ann}_R(N)$, and let $\chi^N_x: N \to N$ defined by
$n\mapsto xn$. This map induced a map
$${\rm Ext}^i_R(M,\chi^N_x):{\rm Ext}^i_R(M,N)\to {\rm Ext}^i_R(M,N)$$
is given by multiplication by $x$.
Assume now that $x\in {\rm Ann}_R(N)$. Then the map $\chi^N_x$
is the zero map, so  implies that the induced map ${\rm Ext}^i_R(M,\chi^N_x)$ is the zero-map. In other words, multiplication by $x$ on ${\rm Ext}^i_R(M,N)$ is zero, as desired.
The case where $x\in {\rm Ann}_R(M)$ is handled similarly, using the map $\chi^M_x$.
\end{proof}

For the next result, remember that a non-zero $R$-module $S$ is called secondary when $S\neq 0$, and for each $r\in R$, either $rS=S$ or there exist $n \in \mathbb{N}$ such that $r^nS=0$. A secondary representation for an $R$-module $M$ is an expression for $M$ as a finite sum of secondary submodules of $M$. If such a representation exists, we say that $M$ is {\it representable}. If $M$ admits a minimal secondary representation $M= S_1+\ldots+ S_n$, then the set ${\rm Att}_R(M)$ of attached primes of $M$ is the set $\{{\rm Rad}(0:_R S_i) \ : i=1,\ldots,n  \}$  \cite[Definition 7.2.2]{grot}. The next result generalizes \cite[Theorem 2.3]{ffdepth1}.

\begin{thm} Let $\mathfrak{a},I,J$ be ideals of $(R,\mathfrak{m})$. Let $M$ be a finitely generated $R$-module. If $\mathfrak{F}^i_{\mathfrak{a},I,J}(M)$ is non-zero and representable for an integer $i$, then $\mathfrak{a}\subset \mathfrak{p}$ for all $\mathfrak{p}\in {\rm Att}_R(\mathfrak{F}^i_{\mathfrak{a},I,J}(M))$.

\end{thm}
\begin{proof}
Let $\mathfrak{F}^i_{\mathfrak{a},I,J}(M)= S_1+\ldots+S_n$ be a minimal secondary representation of $\mathfrak{F}^i_{\mathfrak{a},I,J}(M)$, where $S_k$ is $\mathfrak{p}_k$-secondary for $k=1,\ldots,n$. Suppose by contradiction that $\mathfrak{a}\not\subseteq p_k$ for an integer $k\in \{1,\ldots,n\}$. Then, there is an element $x\in \mathfrak{a}\setminus \mathfrak{p}_j$. We may assume an element $0\neq y=(y_j)\in S_k \subset \mathfrak{F}^i_{\mathfrak{a},I,J}(M)$ with $y_j\in S_k$ be the first non-zero component of $y$.

Note that $xS_K=S_K$ because $x\not\in \mathfrak{p}$. Since $x^jS_k\subseteq x^j\mathfrak{F}^i_{\mathfrak{a},I,J}(M)$, we obtain that $S_k\subseteq x^j\mathfrak{F}^i_{\mathfrak{a},I,J}(M)$. Furthermore $x^jH^i_{I,J}(M/\mathfrak{a}^jM)=0$ by Lemma \ref{lemma0}, then we can conclude that the $j$-th component of each element of $x^j\mathfrak{F}^i_{\mathfrak{a},I,J}(M)$. Therefore, we have the contradiction because $y\in x^j\mathfrak{F}^i_{\mathfrak{a},I,J}(M)$ and $y_j\in S_k$ is non-zero.

\end{proof}
An immediate consequence of the previous results is to show that the formal local cohomology modules with respect to a pair of ideals, when are representable, are $\mathfrak{a}$-torsion modules.
\begin{cor}With the same hypothesis of the previous theorem, if $\mathfrak{F}^i_{\mathfrak{a},I,J}(M)$ is non-zero and representable, for an integer $i$, then $\mathfrak{a}\subseteq {\rm Rad}(0: \mathfrak{F}^i_{\mathfrak{a},I,J}(M)).$
\end{cor}
\begin{proof}Since ${\rm Rad}(0: \mathfrak{F}^i_{\mathfrak{a},I,J}(M))= \cap_{\mathfrak{p}\in {\rm Att}_R(\mathfrak{F}^i_{\mathfrak{a},I,J}(M)) }\mathfrak{p}$ by the same idea of \cite[Proposition 7.2.11]{grot} and by previous theorem
$\mathfrak{a}\subseteq \cap_{\mathfrak{p}\in{\rm Att}_R(\mathfrak{F}^i_{\mathfrak{a},I,J}(M))}\mathfrak{p}$ the statement is true.

\end{proof}

\begin{rem}By previous corollary we can conclude that, if $\mathfrak{F}^i_{\mathfrak{a},I,J}(M))$ is Artinian for all $i>t$ (respectively $i<t$), then $\mathfrak{a}\subseteq {\rm Rad}(0: \mathfrak{F}^i_{\mathfrak{a},I,J}(M))$ for all $i>t$ (respectively $i<t$), since any Artinian $R$-module is representable. Therefore we have that
$$ {\rm sup}\{i:\mathfrak{a} \not\subseteq {\rm Rad}(0: \mathfrak{\check{F}}^i_{\mathfrak{a},\mathfrak{m},J}(M))\}\leq {\rm g{\check g}\!-\!depth}(\mathfrak{a},\mathfrak{m},J,M)\leq {\rm dim} M/(J+\mathfrak{a})M$$
and
$$ 0\leq {\rm ff\!-\!depth}(\mathfrak{a},\mathfrak{m},J,M)\leq {\rm inf}\{i:\mathfrak{a} \not\subseteq {\rm Rad}(0: \mathfrak{\check{F}}^i_{\mathfrak{a},\mathfrak{m},J}(M))\}.$$

\noindent Remember that in this case we are not assuming that $\Gamma_{\mathfrak{a}}(M)$ is an $J$-torsion $R$-module.
The reader can compare this remark with Remark \ref{remark5.}.

%**minha ideia aqui com esse remark era fazer o mesmo resultado do corolary 2.8 novo artigo que te mandei, porem nao consigo fazer a volta! Mesmo assim, consiguimos um resultado interessante, sem a hypotese J-torsion.***

\end{rem}

%** usando a mesma ideia do Brodman 7.3.3 eh possivel mostrar que FLC em $n$ nao eh %finitamente gerado??*** importante

%** outra coisa: sera possivel aplicar a Proposicao 7.2.11 (b) do Brodman em conjunto %do nosso teorema 2.10 para obter alguma caracterizacao em nosso caso?

\noindent In the following we give one of our main results. This result is a generalization of \cite[Theorem 2.9]{131} and \cite[Threom 2.4]{finiteness}.

\begin{thm}\label{artquoc}Let  $\mathfrak{a},J$  be two ideals of $(R,\mathfrak{m})$. Let $M$ be a finitely generated $R$-module such that $\Gamma_{\mathfrak{a}}(M)$ is a $J$-torsion. If for some integer $t$, $\check{\mathfrak{F}}^i_{\mathfrak{a}, \mathfrak{m}, J}(M)$ is Artinian for all $i>t$, then $\check{\mathfrak{F}}^t_{\mathfrak{a}, \mathfrak{m}, J}(M)/ \mathfrak{a}\check{\mathfrak{F}}^t_{\mathfrak{a}, \mathfrak{m}, J}(M)$ is Artinian.
\end{thm}

\begin{proof} We will prove by induction on ${\rm dim}M:=n$. When $n=0$, we have $\check{\mathfrak{F}}^0_{\mathfrak{a}, \mathfrak{m}, J}(M)\cong \mathfrak{F}^0_{\mathfrak{a}, \mathfrak{m}, J}(M)$ \cite[Proposition 4.1]{paper}. Then by Theorem \ref{top}, $\check{\mathfrak{F}}^0_{\mathfrak{a}, \mathfrak{m}, J}(M)$ is Artinian and the result is clear.
Now, we assume $n>0$ and that the claim is true for all values less than $n$. From the exact sequence
$$0\rightarrow \Gamma_{\mathfrak{a}}(M) \rightarrow M\rightarrow M/\Gamma_{\mathfrak{a}}(M) \rightarrow 0,$$
and \cite[Theorem 3.4]{paper} we have the long exact sequence
$$\cdots \rightarrow \mathfrak{\check{F}}^i_{\mathfrak{a},\mathfrak{m},J}(\Gamma_{\mathfrak{a}}(M))\rightarrow \mathfrak{\check{F}}^i_{\mathfrak{a},\mathfrak{m},J}(M)\rightarrow \mathfrak{\check{F}}^i_{\mathfrak{a},\mathfrak{m},J}( M/\Gamma_{\mathfrak{a}}(M))\rightarrow \mathfrak{\check{F}}^{i+1}_{\mathfrak{a},\mathfrak{m},J}(\Gamma_{\mathfrak{a}}(M))\rightarrow \cdots.$$
Since $\Gamma_{\mathfrak{a}}(M)$ is $\mathfrak{a}$-torsion, it  follows that
$$\mathfrak{\check{F}}^i_{\mathfrak{a},\mathfrak{m},J}(\Gamma_{\mathfrak{a}}(M))= H^i(\displaystyle \lim_{ ^{\longleftarrow} }(\check{C}_{\underline{x}, J}\otimes  \Gamma_{\mathfrak{a}}(M) / \mathfrak{a}^n \Gamma_{\mathfrak{a}}(M))\cong H^i_{\mathfrak{m},J}(\Gamma_{\mathfrak{a}}(M)).$$

\noindent By hypothesis, we have that $\mathfrak{\check{F}}^i_{\mathfrak{a},\mathfrak{m},J}( \Gamma_{\mathfrak{a}}(M))\cong H^i_{\mathfrak{m}}(\Gamma_{\mathfrak{a}}(M))$  is Artinian  for all $i$. So, $\mathfrak{\check{F}}^i_{\mathfrak{a},\mathfrak{m},J}( M/\Gamma_{\mathfrak{a}}(M))$ is Artinian for all $i>t$.

Now, we can consider the exact sequence
\begin{equation}\label{equation3}
H^t_{\mathfrak{m}}(\Gamma_{\mathfrak{a}}(M))\rightarrow \mathfrak{\check{F}}^t_{\mathfrak{a},\mathfrak{m},J}(M)\stackrel{\varphi}{\rightarrow} \mathfrak{\check{F}}^t_{\mathfrak{a},\mathfrak{m},J}( M/\Gamma_{\mathfrak{a}}(M))\stackrel{\psi}{\rightarrow} H^{t+1}_{\mathfrak{m}}(\Gamma_{\mathfrak{a}}(M))
\end{equation}

and split to the exact sequences

\begin{equation}\label{equation4}
\begin{array}{l}
0\rightarrow Ker \varphi \rightarrow \mathfrak{\check{F}}^t_{\mathfrak{a},\mathfrak{m},J}(M)\rightarrow Img \varphi \rightarrow 0 \ \ \ \ \ {\rm and}\\
{}\\
0\rightarrow Img \varphi \rightarrow \mathfrak{\check{F}}^i_{\mathfrak{a},\mathfrak{m},J}( M/\Gamma_{\mathfrak{a}}(M))\rightarrow Img \psi \rightarrow 0.
\end{array}
\end{equation}

From these sequences, we deduce the following exact sequences
$$ \frac{Ker \varphi}{\mathfrak{a}Ker \varphi} \rightarrow \frac{\mathfrak{\check{F}}^t_{\mathfrak{a},\mathfrak{m},J}(M)}{\mathfrak{a}\mathfrak{\check{F}}^t_{\mathfrak{a},\mathfrak{m},J}(M)}\rightarrow \frac{Img \varphi}{\mathfrak{a}Img \varphi}  \  \ \ \ \ \ \ \ (1) \ \ \ {\rm and}$$
$$ {\rm Tor}_1^R(R/\mathfrak{a},Img \psi) \rightarrow \frac{Img \varphi}{\mathfrak{a}Img \varphi}\rightarrow \frac{\mathfrak{\check{F}}^t_{\mathfrak{a},\mathfrak{m},J}( M/\Gamma_{\mathfrak{a}}(M))}{\mathfrak{a}\mathfrak{\check{F}}^t_{\mathfrak{a},\mathfrak{m},J}( M/\Gamma_{\mathfrak{a}}(M))}. \ \ \ (2)$$

Since $Ker \varphi$ and $Img \psi$ are Artinian, By the exact sequences (1) and (2), we obtain  that if $\frac{\mathfrak{\check{F}}^t_{\mathfrak{a},\mathfrak{m},J}( M/\Gamma_{\mathfrak{a}}(M))}{\mathfrak{a} \mathfrak{\check{F}}^t_{\mathfrak{a},\mathfrak{m},J}( M/\Gamma_{\mathfrak{a}}(M))}$ is Artinian, then
$\frac{\mathfrak{\check{F}}^t_{\mathfrak{a},\mathfrak{m},J}(M)}{\mathfrak{a}\mathfrak{\check{F}}^t_{\mathfrak{a},\mathfrak{m},J}(M)}$ is Artinian. So, we may assume that $\Gamma_{\mathfrak{a}}(M)=0$.

Let  $x\in \mathfrak{a}\setminus \cup_{\mathfrak{p}\in {\rm Ass}_R M}\mathfrak{p}$. It is known that ${\rm dim}M/xM=n-1$. Then, the short exact sequence
$$0\rightarrow M \stackrel{x}{\rightarrow} M \rightarrow M/xM \rightarrow 0$$
induces the following long exact sequence of formal local cohomology with respect $(\mathfrak{m},J)$.

\begin{equation}\label{equation5}
\cdots \rightarrow \mathfrak{\check{F}}^i_{\mathfrak{a},\mathfrak{m},J}(M)\stackrel{x}{\rightarrow} \mathfrak{\check{F}}^i_{\mathfrak{a},\mathfrak{m},J}(M)\stackrel{\alpha}{\rightarrow} \mathfrak{\check{F}}^i_{\mathfrak{a},\mathfrak{m},J}( M/xM)\stackrel{\beta}{\rightarrow} \mathfrak{\check{F}}^{i+1}_{\mathfrak{a},\mathfrak{m},J}(M)\rightarrow \cdots.
\end{equation}

Since $\mathfrak{\check{F}}^i_{\mathfrak{a},\mathfrak{m},J}(M)$ is Artinian for all $i>t$, we have that $\mathfrak{\check{F}}^i_{\mathfrak{a},\mathfrak{m},J}(M/xM)$ is Artinian for all $i>t$. Then by induction hypothesis, it follows that $\mathfrak{\check{F}}^t_{\mathfrak{a},\mathfrak{m},J}(M/xM)/J\mathfrak{\check{F}}^t_{\mathfrak{a},\mathfrak{m},J}(M/xM)$ is Artinian.
By the previous long exact sequence, we can consider the following short exact sequences
$$0\rightarrow Im \alpha \rightarrow \mathfrak{\check{F}}^t_{\mathfrak{a},\mathfrak{m},J}(M/xM)\rightarrow Im \beta \rightarrow 0 \ \ \ \ \ {\rm and}$$
$$ \mathfrak{\check{F}}^t_{\mathfrak{a},\mathfrak{m},J}(M) \stackrel{x}{\rightarrow} \mathfrak{\check{F}}^t_{\mathfrak{a},\mathfrak{m},J}(M)\rightarrow Im \alpha \rightarrow 0$$
that induces the following two exact sequences
$$ {\rm Tor}_1^R(R/\mathfrak{a},Img \beta) \rightarrow \frac{Img \alpha}{\mathfrak{a}Img \alpha}\rightarrow \frac{\mathfrak{\check{F}}^t_{\mathfrak{a},\mathfrak{m},J}( M/xM)}{\mathfrak{a}\mathfrak{\check{F}}^t_{\mathfrak{a},\mathfrak{m},J}( M/xM)} \ \ \ (3) \ \ \  {\rm and}$$
$$ \frac{\mathfrak{\check{F}}^t_{\mathfrak{a},\mathfrak{m},J}(M)}{\mathfrak{a} \mathfrak{\check{F}}^t_{\mathfrak{a},\mathfrak{m},J}(M)} \stackrel{x}{\rightarrow} \frac{\mathfrak{\check{F}}^t_{\mathfrak{a},\mathfrak{m},J}(M)}{\mathfrak{a}\mathfrak{\check{F}}^t_{\mathfrak{a},\mathfrak{m},J}(M)}\rightarrow \frac{Im \alpha}{\mathfrak{a} Im \alpha} \rightarrow 0. \ \ \ \ \ (4)$$
By choise of $x$, by (2)  follows that $\frac{\mathfrak{\check{F}}^t_{\mathfrak{a},\mathfrak{m},J}(M)}{\mathfrak{a} \mathfrak{\check{F}}^t_{\mathfrak{a},\mathfrak{m},J}(M)}\cong \frac{Im \alpha}{\mathfrak{a} Im \alpha}$. Therefore, since $Im \beta$ is Artinian, $ {\rm Tor}_1^R(R/\mathfrak{a},Img \beta)$ is Artinian. By the exact sequence (3) the proof is completed.
\end{proof}

\begin{cor}\label{cor4.12}With the same hypotheses of the previous result,
if $t= {\rm g\check{g}-{\rm depth}}(\mathfrak{a},\mathfrak{m},J,M)$, then $\check{\mathfrak{F}}^t_{\mathfrak{a}, \mathfrak{m}, J}(M)/ \mathfrak{a}\check{\mathfrak{F}}^t_{\mathfrak{a}, \mathfrak{m}, J}(M)$ is Artinian.

\end{cor}

%\begin{cor}\label{1} Consider $M$ be a finitely generated $R$-module with $\Gamma_{\mathfrak{a}}(M)$ is $J$-torsion $J\neq R$ . If $t= {\rm dim}M/(\mathfrak{a}+J)M$, then $\check{\mathfrak{F}}^t_{\mathfrak{a}, \mathfrak{m}, J}(M)/ \mathfrak{a}\check{\mathfrak{F}}^t_{\mathfrak{a}, \mathfrak{m}, J}(M)$ is Artinian. (??????)
%\end{cor}

%\begin{proof}
%The proof  follow  by Theorem \ref{artquoc}  and authors in \cite[Theorem 4.2 and Remark 4.3]{paper}.
%\end{proof}

\begin{cor}\label{cor4.13}\label{2} Consider $\mathfrak{a}$ and $J$  ideals of $(R,\mathfrak{m})$. Let $M$ be a finitely generated $R$-module such that $\Gamma_{\mathfrak{a}}(M)$ is a $J$-torsion.

\begin{enumerate}
\item [(1)] With the same idea of the proof of Theorem \ref{artquoc} is possible to show that $\check{\mathfrak{F}}^i_{\mathfrak{a}, \mathfrak{m}, J}(M)$ is Artinian for all $i<t$, then $\check{\mathfrak{F}}^t_{\mathfrak{a}, \mathfrak{m}, J}(M)/ \mathfrak{a}\check{\mathfrak{F}}^t_{\mathfrak{a}, \mathfrak{m}, J}(M)$ is Artinian for some  $t$ integer.

\item [(2)] If $t= {\rm f\check{f}-depth}(\mathfrak{a},\mathfrak{m},J,M)$, then $\check{\mathfrak{F}}^t_{\mathfrak{a}, \mathfrak{m}, J}(M)/ \mathfrak{a}\check{\mathfrak{F}}^t_{\mathfrak{a}, \mathfrak{m}, J}(M)$ is Artinian by previous item.

\end{enumerate}
\end{cor}

\section{Coassociated and Cosupport Results}
\hspace{0.5cm}The concept of support and associated primes ideals of local cohomology were studied in great details. With respect to theory dual of this concept, not so much is known for formal local cohomology modules.  The objective of this section is to investigate and give some results about the cosupport and coassociated primes for formal local cohomology modules with respect to a pair of ideals. Firstly  an technical result that improves \cite[Proposition 2.1]{131}.

\textit{}\begin{prop}\label{inter}Let $M$ be an finitely generated $R$-module. Then for all $i\in \mathbb{Z}$   $$\cap_{t>0} \mathfrak{a}^t \mathfrak{F}^i_{\mathfrak{a},\mathfrak{m},J}(M)=0.$$
\end{prop}
\begin{proof} Note that $\mathfrak{a}\displaystyle \varprojlim_{t }N_t\subseteq \displaystyle \varprojlim_{t }\mathfrak{a}N_t$ if $\{N_t \}$ is a inverse system. Therefore

$$
\begin{array}{lll}
\cap_{t>0} \mathfrak{a}^t \mathfrak{F}^i_{\mathfrak{a},\mathfrak{m},J}(M) &\cong & \displaystyle \varprojlim_{t }\mathfrak{a}^t\displaystyle \varprojlim_{ n }H^i_{I,J}(M / \mathfrak{a}^n M) \\
&\subseteq & \displaystyle \varprojlim_{ t }\displaystyle \varprojlim_{ n }\mathfrak{a}^t H^i_{\mathfrak{m},J}(M / \mathfrak{a}^n M)\\
&\cong& \displaystyle \varprojlim_{ n }\displaystyle \varprojlim_{ t }\mathfrak{a}^t H^i_{\mathfrak{m},J}(M / \mathfrak{a}^n M)

\end{array}
$$

Since $H^i_{\mathfrak{m},J}(M / \mathfrak{a}^n M)\cong \displaystyle \varinjlim_{\tilde{\mathfrak{a}}\in \widetilde{W}(\mathfrak{m},J)} H^i_{\tilde{\mathfrak{a}}}(M / \mathfrak{a}^n M)$ \cite[Theorem 3.2]{art1} follows that

$$
\begin{array}{lll}
\cap_{t>0} \mathfrak{a}^t \mathfrak{F}^i_{\mathfrak{a},\mathfrak{m},J}(M) &\cong & \displaystyle \varprojlim_{ n }\displaystyle \varprojlim_{t }\mathfrak{a}^t \displaystyle \varinjlim_{\tilde{\mathfrak{a}}\in \widetilde{W}(\mathfrak{m},J)}  H^i_{\tilde{\mathfrak{a}}}(M / \mathfrak{a}^n M)  \\
&\subseteq & \displaystyle \varprojlim_{n }\displaystyle \varprojlim_{ t }\sum_{\tilde{\mathfrak{a}}\in \widetilde{W}(\mathfrak{m},J)} \mathfrak{a}^t  H^i_{\tilde{\mathfrak{a}}}(M / \mathfrak{a}^n M).

\end{array}
$$

Thus, as for all $t\geq n$  $\mathfrak{a}^t  H^i_{\mathfrak{a}}(M / \mathfrak{a}^n M)=0$ by Lemma \ref{lemma0} the proof is completed.
\end{proof}

\begin{lem}\label{cosup} Let $\mathfrak{a}$ and $J$ be two ideals of  $(R,\mathfrak{m})$, $M$  an $R$-module, not necessarily finitely generated, and $S$ be a multiplicative set of $R$ such that $S\cap \mathfrak{a}\neq \emptyset$. Then ${\rm Hom}_R(S^{-1}R, \mathfrak{F}^i_{\mathfrak{a},\mathfrak{m},J}(M))=0$ for all  integer $i$.
\end{lem}
\begin{proof} Let $s_1$ an element of $S\cap \mathfrak{a}\neq \emptyset$. If $f \in {\rm Hom}_R(S^{-1}R, \mathfrak{F}^i_{\mathfrak{a},\mathfrak{m},J}(M))$ then $f(r/s)= s_1^jf(r/s_1^j s) \in \mathfrak{a}^j\mathfrak{F}^i_{\mathfrak{a},\mathfrak{m},J}(M))$ for all $r/s\in S^{-1}R$ and $j>0$. Then $f(r/s)\in \cap _{j>0} \mathfrak{a}^j\mathfrak{F}^i_{\mathfrak{a},\mathfrak{m},J}(M))=0$ by Proposition \ref{inter}. Therefore $f=0$ and this proof the statement.

\end{proof}
The previous lemma generalizes \cite[Lemma 2.2]{131} and will be useful to characterize the cosupport and coassociated primes of certain formal local cohomologies.  In \cite{melschen}, Melkerson and Schenzel defined the concept the cosupport for an $R$-module $M$ give by
$${\rm CoSupp}(M)=\{ \mathfrak{p}\in {\rm Spec}R \  ; \  {\rm Hom}_R(R_{\mathfrak{p}},M)\neq 0  \}.$$

Our next result improves \cite[Corollary 2.3]{131} and \cite[Corollary 3.5]{majidp}.

\begin{prop}\label{contido} Let $\mathfrak{a}$ and $J$ be two ideals of  $(R,\mathfrak{m})$. Let $M$ be  a finitely generated $R$-module . Then ${\rm CoSupp}(\mathfrak{F}^i_{\mathfrak{a},\mathfrak{m},J}(M))\subset V(\mathfrak a)$ for all integer $i$.
\end{prop}
\begin{proof} Let $\mathfrak{p}\in {\rm CoSupp}(\mathfrak{F}^i_{\mathfrak{a},\mathfrak{m},J}(M))$. Then ${\rm Hom}_R(R_{\mathfrak{p}},\mathfrak{F}^i_{\mathfrak{a},\mathfrak{m},J}(M))\neq 0$ and so, by Lemma \ref{cosup} follow that $\mathfrak{a}\cap (R/\mathfrak{p})=\emptyset$. Therefore $\mathfrak{a}\subset \mathfrak{p}$ and so the claim is proved.

\end{proof}

 A prime ideal $\mathfrak{p}\in R$ is called to be a coassociated prime ideal of  an $R$-module $M$ if there exist a cocyclic homomorphic image $L$ of $M$ such that $\mathfrak{p}= {\rm Ann}(L)$. The set of all coassociated prime ideals of  $M$ is denoted by ${\rm Coass}_R (M)$ (see  \cite{yassemi}). In other words, when we consider an local ring $(R,\mathfrak{m})$, ${\rm Coass}_R (M)= {\rm Ass}(M^{\vee}) $ where $M^{\vee}= {\rm Hom }(M, E(R/\mathfrak{m})).$ Furthermore, if $\mathfrak{a}$ be an ideal of $R$, the following statement is true:

(i)   ${\rm Coass}_R (M/\mathfrak{a}M)= {\rm Coass}_R (M) \cap V(\mathfrak{a})$ \cite[Theorem 1.21]{yassemi};

(ii) If $M$ is Artinian then ${\rm Coass}_R (M)$ is finite \cite[Lemma 1.22]{yassemi};

(iii) $ {\rm Coass}_R M \subset {\rm CoSupp}_R(M)$ \cite[Therem 2.2]{yassemi}.

\begin{thm} Assume $\mathfrak{a}$ be an ideal of $(R, \mathfrak{m})$. Let $M$ be a finitely generated $R$-module with $\rm{dim} M=d $. Then
$${\rm Coass}_R(\mathfrak{F}^d_{\mathfrak{a},I,J}(M))= \{\mathfrak{p}\in {\rm Supp}_R M\cap V(J) \mid {\rm cd}(I,R/\mathfrak{p})=d\}\cap V(\mathfrak{a}),$$
where ${\rm cd}(I,R/\mathfrak{p})$ is the cohomological dimension of the $R$-module $R/\mathfrak{p}$ with respect to $I$. In particular,
$${\rm Coass}_R(\mathfrak{F}^d_{\mathfrak{a},\mathfrak{m},J}(M))= \{\mathfrak{p}\in {\rm Supp}_R M\cap V(J) \mid \dim R/{\mathfrak p}=d\}\cap V(\mathfrak{a}).$$

\end{thm}
\begin{proof}The statement follow immediately by Theorem \ref{top}, Theorem \ref{top1} and \cite[Theorem 1.14]{yassemi}.
\end{proof}
Analogously to Corollary \ref{mesmosuporte}, by previous theorem we have the following result.
\begin{cor}Consider $\mathfrak{a}, I$ and $J$ be ideals of  $(R, \mathfrak{m})$. Let $M$ and $N$ are two finitely generated $R$-modules of dimension $d$ such that ${\rm Supp}_R M= {\rm Supp}_R N$. Then ${\rm Coass}_R(\mathfrak{F}^d_{\mathfrak{a},I,J}(M))= {\rm Coass}_R(\mathfrak{F}^d_{\mathfrak{a},I,J}(N))$.

\end{cor}

The next result generalizes \cite[Corollary 2.5]{finiteness}.
\begin{prop}\label{5.6} Let $\mathfrak{a}$ and $J$ be two ideals of $(R,\mathfrak{m})$ . Let $M$ be a finitely generated $R$-module such that $\Gamma_{\mathfrak{a}}(M)$ is $J$-torsion. If, for some integer $t$, $\check{\mathfrak{F}}^i_{\mathfrak{a}, \mathfrak{m}, J}(M)$ is Artinian for all $i>t$ (respectively $i<t$), then $ {\rm Coass}_R \check{\mathfrak{F}}^t_{\mathfrak{a}, \mathfrak{m}, J}(M) \cap V(\mathfrak{a}) $ is finite.
\end{prop}
\begin{proof}
By Theorem \ref{artquoc} and previous comment (i) and (ii)  the result follows.

\end{proof}

\begin{rem}If $t={\rm g\check{g}-depth}(\mathfrak{a},\mathfrak{m},J,M)$ (resp. $t= {\rm f\check{f}- depth}(\mathfrak{a},\mathfrak{m},J,M)),$ by Corollary \ref{cor4.12} (resp. Corollary  \ref{cor4.13})  follows that $ {\rm Coass}_R \check{\mathfrak{F}}^t_{\mathfrak{a}, \mathfrak{m}, J}(M) \cap V(\mathfrak{a}) $ is finite.

\end{rem}

Note  that, with the same hypothesis of  Proposition \ref{5.6}, by \cite[Proposition 2.4]{paper} and \cite[1.10]{yassemi}, we have that $ {\rm Coass}_R \mathfrak{F}^t_{\mathfrak{a}, \mathfrak{m}, J}(M) \cap V(\mathfrak{a}) $ is finite. Therefore, by previous comment (iii)  and Proposition \ref{contido} we have that $ {\rm Coass}_R  \mathfrak{F}^t_{\mathfrak{a}, \mathfrak{m}, J}(M)$ is finite.

When $M$ is $J$-torsion $R$-module, we have the similar result for formal local cohomology modules with respect to a pair of ideals.

\begin{cor} Let $\mathfrak{a}$ and $J$ be two ideals of $(R,\mathfrak{m}).$ Let $M$ be a finitely generated $R$-module $J$-torsion. If $t= {\rm sup}\{{\rm ara}(\tilde{\mathfrak{a}}) \ ; \ \tilde{\mathfrak{a}}\in \tilde{W}(\mathfrak{m},J) \},$ then $\mathfrak{F}^t_{\mathfrak{a}, \mathfrak{m}, J}(M)\otimes R/\mathfrak{a}$ is Artinian and $ {\rm Coass}_R \mathfrak{F}^t_{\mathfrak{a}, \mathfrak{m}, J}(M) $ is finite.
\end{cor}
\begin{proof} Since $M$ is $J$-torsion we have that $\Gamma_{\mathfrak{a}}(M)$ is $J$-torsion, so we have $\check{\mathfrak{F}}^t_{\mathfrak{a}, \mathfrak{m}, J}(M)\cong \mathfrak{F}^t_{\mathfrak{a}, \mathfrak{m}, J}(M)$ \cite[Corollary 2.6]{paper}. Furthermore, by \cite[Theorem 3.2]{art1},
$$\mathfrak{F}^t_{\mathfrak{a}, \mathfrak{m}, J}(M)\cong  \displaystyle \varprojlim_{n } \displaystyle \varinjlim_{\tilde{\mathfrak{a}}\in \widetilde{W}(\mathfrak{m},J)}  H^t_{\tilde{\mathfrak{a}}}(M / \mathfrak{a}^n M).$$

The statement follows by \cite[Corollary 3.3.3]{grot}, Theorem \ref{artquoc},  Proposition \ref{5.6} and previous comment.

%*** como o limite inverso e direto comutam, sera que podemos tomar o t=infimo pois estaremos com uma intersecao infinita dai?***

\end{proof}

\begin{cor} Let $\mathfrak{a}$ and $J$ be two ideals of $(R,\mathfrak{m}).$ Let  $M$ be a finitely generated $R$-module  $J$-torsion. Then $\mathfrak{F}^{d}_{\mathfrak{a}, \mathfrak{m}, J}(M)\otimes R/\mathfrak{a}$ is Artinian and $ {\rm Coass}_R \mathfrak{F}^{d}_{\mathfrak{a}, \mathfrak{m}, J}(M)$ is finite, where  $d=dim M/(\mathfrak{a}+J)M$.
\end{cor}
\begin{proof}By the same idea of the previous result, we have that $\check{\mathfrak{F}}^t_{\mathfrak{a}, \mathfrak{m}, J}(M)\cong \mathfrak{F}^t_{\mathfrak{a}, \mathfrak{m}, J}(M)$ \cite[Corollary 2.6]{paper}.  The result follows from Proposition \ref{5.6}, Proposition \ref{contido}  and  the fact that $\mathfrak{F}^{j}_{\mathfrak{a}, \mathfrak{m}, J}(M)=0$ for all $d<j$ \cite[Theorem 4.2]{paper}.

\end{proof}

\begin{lem}\label{decomp} Let $\mathfrak{a}$, $I$ and $J$ be ideals of $(R,\mathfrak{m}).$ Let $M$ be a finitely generated $R$-module. Then
$$\mathfrak{F}^{n}_{\mathfrak{a}, I, J}(M)\cong \mathfrak{F}^{n}_{\mathfrak{a}, I, J}(R)\otimes_R M$$
where $n:={\rm dim} R/\mathfrak{a}+J$.
\end{lem}
\begin{proof} By \cite[Theorem 4.2]{paper},  $\mathfrak{F}^{n}_{\mathfrak{a}, I, J}(-)$ is a right exact functor and by definition of inverse limit, $\mathfrak{F}^{i}_{\mathfrak{a}, I, J}(-)$ preserves finite direct sum. The statement follows by Watt's Theorem \cite[Theorem 5.45]{brod}.

\end{proof}

By this result, for finitely generated $R$-modules $M$,  we can see that $\mathfrak{F}^{n}_{\mathfrak{a}, I, J}(M)=0$ if and only if $\mathfrak{F}^{n}_{\mathfrak{a}, I, J}(R)=0$.

The \cite[Lemma 4.2]{majid} showed that if $R$  a local ring  and $M$ be an $R$-module, the set of minimal primes of ${\rm CoSupp}_R(M)$ is finite if, and only if, ${\rm CoSupp}_R(M)$ is a closed subset of ${\rm Spec}R$. So, by \cite[Theorem 2.6]{yassemi} we can conclude that ${\rm CoSupp}_R(\mathfrak{F}^{i}_{\mathfrak{a}, I, J}(M))$ is closed if and only if ${\rm Coass}_R(\mathfrak{F}^{i}_{\mathfrak{a}, I, J}(M))$ is finite, for some integer $i$. By Theorem \ref{top} we have that ${\rm CoSupp}_R(\mathfrak{F}^{d}_{\mathfrak{a}, I, J}(M))$ is closed, where ${\rm dim}M=d$.

\begin{rem} Let $(R,\mathfrak{m})$ be a Cohen Macaulay complete ring of dimension $d$ and $J$ be a perfect ideal of $R$, i.e, ${\rm grade}(J,R)= {\rm pd}_R R/J=t$. Then, if $M$ is a finitely generated $R$-module, by construction in Section \ref{section3} we have the isomorphism
$$ \mathfrak{F}^i_{\mathfrak{a},\mathfrak{m},J}(M)\cong H^{d-t-i}_{\mathfrak{a}}(M,S)^\vee$$ where $(-)^\vee= {\rm Hom}_R(-,E_R(\mathbb{K}))$, $S= H^{d-t}_{\mathfrak{m},J}(R)^\vee$, $i \in \mathbb{Z}$.

\noindent  Therefore, by \cite[Corollary 1.18]{yassemi} (respectively \cite[Corollary 2.9]{yassemi}) we obtain that, if $H^{d-t-i}_{\mathfrak{a}}(M,S)$ is a finite module, then
$${\rm CoSupp}_R(\mathfrak{F}^{i}_{\mathfrak{a}, \mathfrak{m}, J}(M))= {\rm Supp}_R (H^{d-t-i}_{\mathfrak{a}}(M,S) )  \ \ \  {\rm and}$$
$${\rm Coass}_R(\mathfrak{F}^{i}_{\mathfrak{a}, \mathfrak{m}, J}(M))= {\rm Ass}_R (H^{d-t-i}_{\mathfrak{a}}(M,S) ).$$

%( eh finitamente gerado esse modulo CL em S???***ver direito se no resultao anterior esta enunciado certo)

\end{rem}

Our next result generalizes \cite[Proposition 4.4]{majidp}.
\begin{prop}Let $(R,\mathfrak{m})$ be a Cohen Macaulay complete ring of dimension $d$ and $J$ be a perfect ideal of $R$, i.e, ${\rm grade}(J,R)= {\rm pd}_R R/J=t.$ If $M$ is a finitely generated $R$-module, then
$${\rm Coass}_R(\mathfrak{F}^{n}_{\mathfrak{a}, \mathfrak{m}, J}(M))= {\rm Supp}_R(M)\cap {\rm Ass}_R (H^{d-t-n}_{\mathfrak{a}}(S) )$$
where $n:={\rm dim} R/\mathfrak{a}+J$.

\end{prop}
\begin{proof} By Lemma \ref{decomp}, \cite[Theorem 1.21]{yassemi} and previous comment follow that
$$
\begin{array}{lll}
{\rm Coass}_R(\mathfrak{F}^{n}_{\mathfrak{a}, \mathfrak{m}, J}(M))& = & {\rm Coass}_R(\mathfrak{F}^{n}_{\mathfrak{a}, \mathfrak{m}, J}(R)\otimes_R M)\\
&=& {\rm Supp}_R(M)\cap {\rm Coass}_R(\mathfrak{F}^{n}_{\mathfrak{a}, \mathfrak{m}, J}(R))\\
&=&{\rm Supp}_R(M)\cap {\rm Ass}_R (H^{d-t-n}_{\mathfrak{a}}(S) ).
\end{array}
$$

\end{proof}

%The another consequence of previous result is that, by Corollary \ref{1} (respectively Remark \ref{2})  if
%$t= {\rm dim}M/(\mathfrak{a}+J)M$ (respectively $t= {\rm fgrade}(\mathfrak{a},\mathfrak{m},J)$) then $ {\rm Coass}_R \check{\mathfrak{F}}^t_{\mathfrak{a}, \mathfrak{m}, J}(M) \cap V(\mathfrak{a})$ is finite. (???? acho que os remarks estao errados mas podemos fazer como no corolario 6.3- dizer que M é J torcao nao é fraco nesse caso, eu te explico porque Victor**)

%With this, we can conclude that if
%$t= {\rm dim}M/(\mathfrak{a}+J)M$ (respectively $t= {\rm fgrade}(\mathfrak{a},\mathfrak{m},J)$) then $ {\rm Coass}_R \mathfrak{F}^t_{\mathfrak{a}, \mathfrak{m}, J}(M)$ is finite. (acho que não esta correto por causa dos Remarks 5.7 e 5.8 ***)

%Corollary 2.5 Finiteness and 2.10 Mafi

\section{Finitenness of $\mathfrak{\check{F}}^i_{\mathfrak{a},\mathfrak{m},J}(M)$ and $\mathfrak{F}^i_{\mathfrak{a},\mathfrak{m},J}(M)$}

\hspace{0.5cm} In this section, we investigate the finiteness of \v{C}ech formal local cohomology and formal local cohomology modules with respect to a pair of ideals. We give an important criterion for finiteness, and we show an especific case of non-finiteness for  \v{C}ech formal local cohomology and formal local cohomology modules with respect to a pair of ideals.  The next result extends \cite[Theorem 2.8]{yan}.

\begin{thm}\label{final}
Let  $\mathfrak{a},J$ be two ideals of $(R,\mathfrak{m})$. Let $M$ be a finitely generated $R$-module and $t\geq 1$ be an integer. The following statements are equivalent:

\begin{itemize}
\item[(1)] $\mathfrak{\check{F}}^i_{\mathfrak{a},\mathfrak{m},J}(M)=0$ for all $i\geq t$;
\item[(2)] $\mathfrak{\check{F}}^i_{\mathfrak{a},\mathfrak{m},J}(M)$ is finitely generated for all $i\geq t$;
\item[(3)] $\mathfrak{\check{F}}^i_{\mathfrak{a},\mathfrak{m},J}(R/\mathfrak{p})= 0$ for all $i\geq t$, $\mathfrak{p}\in {\rm Supp} M$;
\item[(4)] $\mathfrak{\check{F}}^i_{\mathfrak{a},\mathfrak{m},J}(R/\mathfrak{p})$ is finitely generated for all $i\geq t$, $\mathfrak{p}\in {\rm Supp} M$.
\end{itemize}
\end{thm}
\begin{proof} $(1)\Rightarrow (2)$ is clear.

$(2)\Rightarrow (1)$ Consider $d= {\rm dim} M \geq \dim (M/{\mathfrak a}M)$ and we argue by induction on $d$. If $d=0$, we have that $\mathfrak{\check{F}}^i_{\mathfrak{a},\mathfrak{m},J}(M)\cong \mathfrak F^i_{\mathfrak{a},\mathfrak{m},J}(M)=0$ for all $i\geq 1$ \cite[Proposition 4.1]{paper}.
So, suppose that $d>0$ and the result has been proved for smaller values of $d$. Firstly, we assume ${\rm depth} M>0$. Thus, there exists an $M$-regular element $x$ in $\mathfrak{m}$. The short exact sequence
$$0\rightarrow M\stackrel{x}{\rightarrow} M\rightarrow M/xM\rightarrow 0$$
induces the next long exact sequence of formal local cohomology
$$\cdots \rightarrow \mathfrak{\check{F}}^i_{\mathfrak{a},\mathfrak{m},J}(M)\stackrel{x}{\rightarrow} \mathfrak{\check{F}}^i_{\mathfrak{a},\mathfrak{m},J}(M)\rightarrow \mathfrak{\check{F}}^i_{\mathfrak{a},\mathfrak{m},J}( M/xM)\rightarrow \mathfrak{\check{F}}^{i+1}_{\mathfrak{a},\mathfrak{m},J}(M)\rightarrow \cdots.$$

By the inductive hypothesis, we have $\mathfrak{\check{F}}^i_{\mathfrak{a},\mathfrak{m},J}( M/xM)=0$ for all $i\geq t$ e then $x\mathfrak{\check{F}}^i_{\mathfrak{a},\mathfrak{m},J}( M)= \mathfrak{\check{F}}^i_{\mathfrak{a},\mathfrak{m},J}( M)$ for all $i\geq t$. Since $\mathfrak{\check{F}}^i_{\mathfrak{a},\mathfrak{m},J}( M)$ is finitely generated  for all $i\geq t$, we have $\mathfrak{\check{F}}^i_{\mathfrak{a},\mathfrak{m},J}( M)=0$ for all $i\geq t$.

Now, assume ${\rm depth}M=0$ and consider  $N=H_{\mathfrak{m}}^0(M)$. Note that
$$\mathfrak{\check{F}}^0_{\mathfrak{a},\mathfrak{m},J}( N)\cong \mathfrak F^0_{\mathfrak{a},\mathfrak{m},J}(N)=\displaystyle \varprojlim_{ n }H^0_{\mathfrak{m},J}(N/\mathfrak{a}^n N)= N,$$
and $\mathfrak{\check{F}}^i_{\mathfrak{a},\mathfrak{m},J}( N)=0$ for all $i\geq 1$ because $N$ is too $(\mathfrak{m},J)$ torsion and Artinian $R$-module. From the short exact sequence $0\rightarrow N\rightarrow M\rightarrow M/N\rightarrow 0$ we obtain that $\mathfrak{\check{F}}^i_{\mathfrak{a},\mathfrak{m},J}(M)\cong \mathfrak{\check{F}}^i_{\mathfrak{a},\mathfrak{m},J}(M/N)$ for all $i\geq 1$. By this, we may assume that $M$ is $\mathfrak{m}$-torsion free and the statement follows by the first argument.

$(1)\Rightarrow (3)$ By \cite[Proposition 2.4 and Theorem 4.2]{paper} we have that
$$\dim_R M/(\mathfrak{a}+J)M= {\rm sup} \{i \in \mathbb{Z} \ | \ \mathfrak{F}^i_{\mathfrak{a},\mathfrak{m},J}(M)\neq 0\}\leq {\rm sup} \{i \in \mathbb{Z} \ | \ \mathfrak{\check{F}}^i_{\mathfrak{a},\mathfrak{m},J}(M)\neq 0\}.$$
Since ${\rm dim}\frac{R}{(\mathfrak{a}+J)+\mathfrak{p}}\leq \dim\frac{M}{(\mathfrak{a}+J)M} $ for all $\mathfrak{p}\in {\rm Supp}M$, we can conclude that ${\rm sup} \{i \in \mathbb{Z} \ | \ \mathfrak{\check{F}}^i_{\mathfrak{a},\mathfrak{m},J}(R/\mathfrak{p})\neq 0\}\leq {\rm sup} \{i \in \mathbb{Z} \ | \ \mathfrak{\check{F}}^i_{\mathfrak{a},\mathfrak{m},J}(M)\neq 0\}.$ Therefore, $\mathfrak{\check{F}}^i_{\mathfrak{a},\mathfrak{m},J}(R/\mathfrak{p})=0$ for all $i\geq t$.

$(3)\Rightarrow (1)$ Consider a prime filtration $0=M_0\subseteq M_1\subseteq \ldots \subseteq M_s=M$ of submodules de $M$ such that $M_j/M_{j-1}\cong R/\mathfrak{p}_j$ where $p_j \in {\rm Supp M}$ and $1\leq j\leq s$. By induction on $j$, and the exact sequence
$$\mathfrak{\check{F}}^i_{\mathfrak{a},\mathfrak{m},J}(M_{j-1})\rightarrow \mathfrak{\check{F}}^i_{\mathfrak{a},\mathfrak{m},J}(M_j)\rightarrow \mathfrak{\check{F}}^i_{\mathfrak{a},\mathfrak{m},J}( R/\mathfrak{p}_j)$$ follows that $ \mathfrak{\check{F}}^i_{\mathfrak{a},\mathfrak{m},J}(M)=0$ for all $i\geq t.$

$(3)\Leftrightarrow (4)$ It is analogous the previous proof of $(1)\Leftrightarrow (2)$ and this finishes the result.

\end{proof}

An immediate consequence of the previous result is the next corollary.

\begin{cor} Let $\dim_R M/(\mathfrak{a}+J)M = d >0$ with $J\neq R$. Then $\mathfrak{\check{F}}^d_{\mathfrak{a},\mathfrak{m},J}(M)$ is not finitely generated.
\end{cor}
\begin{proof} If $\mathfrak{\check{F}}^d_{\mathfrak{a},\mathfrak{m},J}(M)$ is finitely generated, then  $\mathfrak{\check{F}}^i_{\mathfrak{a},\mathfrak{m},J}(M)$ is finitely generated for all $i\geq d$. Therefore, by previous theorem $\mathfrak{\check{F}}^i_{\mathfrak{a},\mathfrak{m},J}(M)=0$ for all $i\geq d$. Remember that
$d= {\rm sup} \{i \in \mathbb{Z} \ | \ \mathfrak{F}^i_{\mathfrak{a},\mathfrak{m},J}(M)\neq 0\}\leq {\rm sup} \{i \in \mathbb{Z} \ | \ \mathfrak{\check{F}}^i_{\mathfrak{a},\mathfrak{m},J}(M)\neq 0\}$. Since $\mathfrak{\check{F}}^d_{\mathfrak{a},\mathfrak{m},J}(M)=0$, by \cite[Proposition 2.4]{paper} we have  that $\mathfrak{F}^d_{\mathfrak{a},\mathfrak{m},J}(M)=0$ and this is a contradiction.

\end{proof}

With respect the relation of finiteness of the formal local cohomology with respect to a pair of ideals, we give the following result.

\begin{cor}Let $\mathfrak{a},J$ be two ideals of $(R,\mathfrak{m})$. Let $M$ be a finitely generated $R$-module and $t\geq 1$ be an integer. The following statements are equivalent:

\begin{itemize}
\item[(1)] $\mathfrak{F}^i_{\mathfrak{a},\mathfrak{m},J}(M)=0$ for all $i\geq t$;
\item[(2)] $\mathfrak{F}^i_{\mathfrak{a},\mathfrak{m},J}(M)$ is finitely generated for all $i\geq t$;
\item[(3)] $\mathfrak{F}^i_{\mathfrak{a},\mathfrak{m},J}(R/\mathfrak{p})= 0$ for all $i\geq t$, $\mathfrak{p}\in {\rm Supp} M$;
\item[(4)] $\mathfrak{F}^i_{\mathfrak{a},\mathfrak{m},J}(R/\mathfrak{p})$ is finitely generated for all $i\geq t$, $\mathfrak{p}\in {\rm Supp} M$.

\end{itemize}
Furthermore, if $\dim_R M/(\mathfrak{a}+J)M = d >0$ with $J\neq R$, then $\mathfrak{F}^d_{\mathfrak{a},\mathfrak{m},J}(M)$ is not finitely generated.
\end{cor}
\begin{proof} This result follows by Theorem \ref{final} and the exact sequence in \cite[Proposition 2.4]{paper}. The second statement follows by the same idea of previous corollary.
\end{proof}

\begin{thm} Let $\mathfrak{a}$ and $J$ be two ideals of $(R,\mathfrak{m})$. Let $M$ be a finitely generated $R$-module such that  $\Gamma_{\mathfrak{a}}(M)$ is $J$-torsion. Then ${\rm Hom}(R/\mathfrak{m}, \mathfrak{\check{F}}^{t}_{\mathfrak{a},\mathfrak{m},J}(M))$ is finitely generated, where $t={\rm f{\check f}-depth}(\mathfrak{a},\mathfrak{m},J,M).$

\end{thm}
\begin{proof}
Use induction on $t$. When $t=0$, since $\mathfrak{\check{F}}^0_{\mathfrak{a},\mathfrak{m},J}(M)\cong \mathfrak{\check{F}}^0_{\mathfrak{a}{\hat R},\mathfrak{m}{\hat R},J{\hat R}}({\hat M})$ and $\dim (M/{\mathfrak a}M)=\dim({\hat M}/{\mathfrak a}{\hat R}{\hat M})$, we may assume that $M$ is complete in $\mathfrak m$-adic topology. So, $M$ is also complete in $\mathfrak a$-adic topology. Hence $\mathfrak{\check{F}}^0_{\mathfrak{a},\mathfrak{m},J}(M)=H^0(\displaystyle \lim_{ ^{\longleftarrow} }(\check{C}_{\underline{x}, J}\otimes  M / \mathfrak{a}^n M))\subset \lim_{ ^{\longleftarrow} }(M/{\mathfrak a}^nM)=M.$ This implies that $\mathfrak{\check{F}}^0_{\mathfrak{a},\mathfrak{m},J}(M)$  is finitely generated as an $\hat R$-module, then ${\rm Hom}({\hat R}/\mathfrak{{\hat m}}, \mathfrak{\check{F}}^0_{\mathfrak{a},\mathfrak{m},J}(M))$ is finitely generated as a $\hat R$-module. Moreover ${\rm Hom}({\hat R}/\mathfrak{{\hat m}}, \mathfrak{\check{F}}^0_{\mathfrak{a},\mathfrak{m},J}(M))$ is a $R/{\mathfrak m}$-vector space of finite dimension.

\noindent On the other hand, since $ \mathfrak{\check{F}}^0_{\mathfrak{a},\mathfrak{m},J}(M)$ has a structure as an $\hat R$-module, we have the following isomorphism:
$$
\begin{array}{lll}
{\rm Hom}_{\hat R}({\hat R}/\mathfrak{{\hat m}}, \mathfrak{\check{F}}^0_{\mathfrak{a},\mathfrak{m},J}(M))&\cong & {\rm Hom}_{\hat R}(R/\mathfrak{m}\otimes{\hat R}, \mathfrak{\check{F}}^0_{\mathfrak{a},\mathfrak{m},J}(M))\\
&\cong & {\rm Hom}_R(R/\mathfrak{m}, {\rm Hom}_{\hat R}({\hat R}, \mathfrak{\check{F}}^0_{\mathfrak{a},\mathfrak{m},J}(M)))\\
&\cong & {\rm Hom}_R(R/\mathfrak{m}, \mathfrak{\check{F}}^0_{\mathfrak{a},\mathfrak{m},J}(M)).
\end{array}
$$

\noindent Thus, ${\rm Hom}_R(R/\mathfrak{m}, \mathfrak{\check{F}}^0_{\mathfrak{a},\mathfrak{m},J}(M))$ is finitely generated.

 Next, we suppose that $t>0$ and that the claim has been proved for smaller values of $t$. From the exact sequence

 $$0\rightarrow  \Gamma_{\mathfrak{a}}(M)\rightarrow M \rightarrow M/\Gamma_{\mathfrak{a}}(M)\rightarrow 0,$$

\noindent by \cite[Theorem 3.4]{paper} we have the long exact sequence
$$\cdots \rightarrow \mathfrak{\check{F}}^i_{\mathfrak{a},\mathfrak{m},J}(\Gamma_{\mathfrak{a}}(M))\rightarrow \mathfrak{\check{F}}^i_{\mathfrak{a},\mathfrak{m},J}(M)\rightarrow \mathfrak{\check{F}}^i_{\mathfrak{a},\mathfrak{m},J}( M/\Gamma_{\mathfrak{a}}(M))\rightarrow \mathfrak{\check{F}}^{i+1}_{\mathfrak{a},\mathfrak{m},J}(\Gamma_{\mathfrak{a}}(M))\rightarrow \cdots.$$

Also, by assumption we have

$$
\begin{array}{lll}
\mathfrak{\check{F}}^i_{\mathfrak{a},\mathfrak{m},J}(\Gamma_{\mathfrak{a}}(M))&=& H^i(\displaystyle \lim_{ ^{\longleftarrow} }(\check{C}_{\underline{x}, J}\otimes  \Gamma_{\mathfrak{a}}(M) / \mathfrak{a}^n \Gamma_{\mathfrak{a}}(M))\\
&\cong & H^i_{\mathfrak{m},J}(\Gamma_{\mathfrak{a}}(M))\\
&\cong & H^i_{\mathfrak{m}}(\Gamma_{\mathfrak{a}}(M)),
\end{array}
$$
\noindent it follows that for some integer $s$, $\mathfrak{\check{F}}^i_{\mathfrak{a},\mathfrak{m},J}(M)$ is Artinian for all $i<s$ if, and only if,  $\mathfrak{\check{F}}^i_{\mathfrak{a},\mathfrak{m},J}( M/\Gamma_{\mathfrak{a}}(M))$.

\noindent From the exact sequences (\ref{equation3}) and (\ref{equation4}), we have the following exact sequence
$$
\cdots \rightarrow {\rm Hom}_R(R/{\mathfrak m},{\rm Ker}{\alpha})\rightarrow {\rm Hom}_R(R/{\mathfrak m},\mathfrak{\check{F}}^i_{\mathfrak{a},\mathfrak{m},J}(M)) \rightarrow {\rm Hom}_R(R/{\mathfrak m},{\rm Im}{\alpha})
$$
$$\rightarrow {\rm Ext}^1_R(R/{\mathfrak m},{\rm Ker}{\alpha}) \rightarrow\cdots,$$

$$
\cdots \rightarrow {\rm Hom}_R(R/{\mathfrak m},{\rm Im}{\beta})\rightarrow {\rm Hom}_R(R/{\mathfrak m},\mathfrak{\check{F}}^i_{\mathfrak{a},\mathfrak{m},J}( M/\Gamma_{\mathfrak{a}}(M))) \rightarrow {\rm Hom}_R(R/{\mathfrak m},{\rm Im}{\beta}).
$$
\noindent Note that ${\rm Ker}\alpha$ and ${\rm Im}\beta$ are Artinian, then the above sequences yiels that ${\rm Hom}_R(R/{\mathfrak m},\mathfrak{\check{F}}^i_{\mathfrak{a},\mathfrak{m},J}(M))$ is finitely generated if and only if \\${\rm Hom}_R(R/{\mathfrak m},\mathfrak{\check{F}}^i_{\mathfrak{a},\mathfrak{m},J}( M/\Gamma_{\mathfrak{a}}(M)))$. Therefore, we may assume that $\Gamma_{\mathfrak{a}}(M)=0$. There exist an $M$-regular element $x\in {\mathfrak a}$.

 \noindent By \cite[Theorem 3.4]{paper}, the short exact sequence
 $$0 \rightarrow M\stackrel{x}{\rightarrow} M\rightarrow M/xM\rightarrow  0$$ induces the long exact sequence (\ref{equation5}). It yields that $\mathfrak{\check{F}}^{t-1}_{\mathfrak{a},\mathfrak{m},J}(M/xM)$ is Artinian for $i<t-1$. So by the inductive hypothesis, ${\rm Hom}_R(R/{\mathfrak m},\mathfrak{\check{F}}^{t-1}_{\mathfrak{a},\mathfrak{m},J}(M/xM))$  is finitely generated. On the other hand, the exact sequence

 $$0 \rightarrow \frac{\mathfrak{\check{F}}^{t-1}_{\mathfrak{a},\mathfrak{m},J}(M)}{x\mathfrak{\check{F}}^{t-1}_{\mathfrak{a},\mathfrak{m},J}(M)} \rightarrow \mathfrak{\check{F}}^{t-1}_{\mathfrak{a},\mathfrak{m},J}(M/xM)\rightarrow (0:_{\mathfrak{\check{F}}^t_{\mathfrak{a},\mathfrak{m},J}(M)}x)\rightarrow  0$$

 \noindent induces the exact sequence

$
\begin{array}{lll}
{\rm Hom}_R(R/{\mathfrak m},\mathfrak{\check{F}}^{t-1}_{\mathfrak{a},\mathfrak{m},J}(M/xM))& \rightarrow &{\rm Hom}_R(R/{\mathfrak m},(0:_{\mathfrak{\check{F}}^t_{\mathfrak{a},\mathfrak{m},J}(M)}x))\\
&&\rightarrow {\rm Ext}^1_R\left(R/{\mathfrak m},\frac{\mathfrak{\check{F}}^{t-1}_{\mathfrak{a},\mathfrak{m},J}(M)}{x\mathfrak{\check{F}}^{t-1}_{\mathfrak{a},\mathfrak{m},J}(M)}\right)
\end{array}
$

\noindent Since $\frac{\mathfrak{\check{F}}^{t-1}_{\mathfrak{a},\mathfrak{m},J}(M)}{x\mathfrak{\check{F}}^{t-1}_{\mathfrak{a},\mathfrak{m},J}(M)}$ is Artinian, ${\rm Ext}^1_R\left(R/{\mathfrak m},\frac{\mathfrak{\check{F}}^{t-1}_{\mathfrak{a},\mathfrak{m},J}(M)}{x\mathfrak{\check{F}}^{t-1}_{\mathfrak{a},\mathfrak{m},J}(M)}\right)$ is finitely generated by  \cite[Lemma 2.2]{Chu-Lizhong}.

\noindent Then we obtain that ${\rm Hom}_R(R/{\mathfrak m},(0:_{\mathfrak{\check{F}}^t_{\mathfrak{a},\mathfrak{m},J}(M)}x))$ is finitely generated. Since $x\in \mathfrak m$, we have

$$
\begin{array}{lll}
{\rm Hom}_R(R/{\mathfrak m},(0:_{\mathfrak{\check{F}}^t_{\mathfrak{a},\mathfrak{m},J}(M)}x))&\cong & {\rm Hom}_R(R/{\mathfrak m}\otimes R/xR,{\check{\mathfrak F}}^t_{\mathfrak{a},\mathfrak{m},J}(M))\\
&\cong & {\rm Hom}_R(R/{\mathfrak m},{\check{\mathfrak F}}^t_{\mathfrak{a},\mathfrak{m},J}(M))
\end{array}
$$

\noindent therefore ${\rm Hom}_R(R/{\mathfrak m},{\check{\mathfrak F}}^t_{\mathfrak{a},\mathfrak{m},J}(M))$ is finitely generated.

\end{proof}

\noindent {\it Acknowledgements:}  The authors would like to thank by Professors B. Ulrich and G. Caviglia for some conversations, and by hospitality for the Department of Mathematics-Purdue University.

\end{document}